\documentclass[11pt, a4paper]{amsart}

\usepackage{amsmath, amsthm, amssymb,fullpage}
\usepackage{amsmath}
\usepackage{amsfonts}
\usepackage{amssymb}
\usepackage{amsthm}

\usepackage{color}
\usepackage{xcolor}

\usepackage{enumitem}
\usepackage{hyperref}

\newcommand{\bD}{\mathbb{D}}
\newcommand{\bC}{\mathbb{C}}

\theoremstyle{plain}
\newtheorem{theorem}{Theorem}[section]
\newtheorem{corollary}[theorem]{Corollary}
\newtheorem{lemma}[theorem]{Lemma}
\newtheorem{proposition}[theorem]{Proposition}
\newtheorem{claim}{Claim}

\theoremstyle{definition}
\newtheorem{definition}[theorem]{Definition}

\theoremstyle{remark}
\newtheorem{remark}[theorem]{Remark}

\numberwithin{equation}{section}

\newcommand{\C}{\mathbb C}
\newcommand{\R}{\mathbb R}
\newcommand{\lam}{\lambda}

\newcommand{\Cc}{\mathcal C}

\newcommand{\Rc}{\mathcal R}

\newcommand{\Sc}{\mathcal S}
\usepackage{bbold}

\newcommand{\D}{\mathbb D}
\newcommand{\N}{\mathbb N}
\newcommand{\B}{\mathbb B}

\usepackage{tikz-cd}

\DeclareMathOperator{\dist}{dist}
\DeclareMathOperator{\supp}{supp}

\newcommand{\ddc}{{dd^c}}

\DeclareMathOperator{\DSH}{DSH}

\usepackage{enumitem}
\usepackage{hyperref}

\begin{document}

\hyphenpenalty=10000

\title[]{Monotonicity of dynamical degrees for H\'enon-like and polynomial-like  maps}

\begin{author}[F.~Bianchi]{Fabrizio Bianchi}
\address{
CNRS, Univ. Lille, UMR 8524 - Laboratoire Paul Painlev\'e, F-59000 Lille, France}
  \email{fabrizio.bianchi$@$univ-lille.fr}
\end{author}

\begin{author}[T.C.~Dinh]{Tien-Cuong Dinh}
\address{National University of Singapore, Lower Kent Ridge Road 10,
Singapore 119076, Singapore}
\email{matdtc$@$nus.edu.sg }
\end{author}

\begin{author}[K.~Rakhimov]{Karim Rakhimov}

\address{National University of Singapore, Lower Kent Ridge Road 10,
Singapore 119076, Singapore}
\email{rkarim@nus.edu.sg }
\end{author}

\begin{abstract}
We prove that, for every invertible horizontal-like map (i.e.,
H\'enon-like map) in any dimension,
the sequence of the dynamical degrees
is increasing
until that of maximal value, which is the main dynamical degree, and decreasing after that. 
Similarly,
for
polynomial-like maps in any dimension, the sequence of dynamical degrees is increasing
until the last one, which is the topological degree.
This is the first time that such a property is proved outside of the algebraic setting. Our proof is based on  the construction of a suitable deformation for positive closed currents, which relies on tools
from pluripotential theory and the solution of the $d, \bar \partial$,
and $dd^c$ equations on convex domains.
\end{abstract}

\maketitle

\noindent
{\bf Notation.} $\D_r$ and $\D$ denote the disc of
radius $r$
and centre 0 and the unit disc in $\C$, respectively.
For $k\geq 1$,
$\mathbb B= \mathbb B_k$ denotes the unit ball in $\mathbb C^k$.
$M$ and $N$ will denote open bounded convex subsets of $\mathbb C^p$ and
$\mathbb C^{k-p}$, respectively, for some
fixed $1\leq p\leq k$.
We will usually 
denote by $M',M''$ (resp.\ $N', N''$) open bounded convex subsets of $\mathbb C^p$ (resp.\ $\mathbb C^{k-p}$) which are slightly smaller than $M$ (resp. $N$), 
with $M''\Subset M'\Subset M$
and $N''\Subset N'\Subset N$,
and set $D':=M'\times N'$ and $D'':= M'' \times N''$.
 We also denote by  $M^\star\subset \mathbb C^p$  and $N^\star\subset\mathbb C^{k-p}$  further auxiliary
 convex open sets satisfying $M''\Subset M^\star\subseteq M$ and  $N''\Subset N^\star\subseteq N$.
  For $p=k$, $N$ reduces to a point and we take $N''=N'=N$.

The definition of a horizontal-like map
$f$
on $M\times N$
is given in Definition \ref{d:HLM}, see also Remark \ref{r:PL} for the special case of polynomial-like maps (corresponding to $p=k$).
For such maps, 
the dynamical degrees 
 of type I
$\lam^\pm_{s}$
and 
 of type II 
$d^\pm_{s}$ are defined
in Definitions \ref{d:degrees-smooth} 
and \ref{d:degrees-currents}, respectively\footnote{The dynamical degrees $d_s^\pm$ were introduced in \cite{DNS} for 
invertible horizontal-like maps.
 In the case of polynomial-like maps, 
the dynamical degrees were introduced in \cite{DSallure} and \cite{DS10}, and denoted by
$d_s$ and $d^*_s$, respectively. We choose here a notation to avoid confusion between them.
}.
For $0\leq s \leq p$ 
(resp. $0\leq s \leq k-p$) 
we  denote by $\mathcal H_s(M^\star\times N^\star)$
(resp. $\mathcal V_{s}(M^\star\times N^\star)$)
the space
of  horizontal (resp. vertical)
positive closed currents of bi-dimension $(s,s)$ 
on $M^\star\times N^\star$
and of finite mass.
We will denote by $\mathcal H_s^{(1)}(M^\star\times N^\star)$
and $\mathcal V_{s}^{(1)}(M^\star\times N^\star)$ the subsets of $\mathcal H_s(M^\star\times N^\star)$ and $\mathcal V_{s}(M^\star\times N^\star)$ given by currents of mass 1,
respectively.
 For $0\leq s \leq p$  the semi-distance
$\dist_{M^\star\times N^\star}$ 
on 
$\mathcal{H}_s(M^\star\times N^\star) $ 
is defined
in \eqref{eq:dist}. A similar semi-distance is defined on
$\mathcal V_{s}(M^\star\times N^\star)$ for
 $0\leq s \leq k-p$.

The pairing $\langle \cdot, \cdot\rangle$
is used for the integral of a function with respect to a measure or more
generally the value of a current at a test form.
By $(s,s)$-forms and $(s,s)$-currents, we mean forms and currents of bi-degree $(s,s)$, respectively.
 The mass of a positive 
 (resp. negative)
 $(s,s)$-current $R$
 on an open subset $U \subset \C^k$
is $\|R\|_U:=\int_U R\wedge \omega^{k-s}$
(resp. $\|R\|_U:=-\int_U R\wedge \omega^{k-s}$),
where $\omega: = dd^c \|z\|^2$ for $z\in \mathbb C^k$
is the standard K\"ahler form on $\C^k$.
Recall that $d^c = \frac{i}{2\pi}(\bar \partial - \partial)$
and $\ddc = \frac{i}{\pi} \partial \bar \partial$. 
 The definition generalizes to positive and negative currents on compact K\"ahler manifolds.
 
  The 
  notations $\lesssim$
  and $\gtrsim$ stand for inequalities up to
a multiplicative constant (usually depending on the domains under consideration). 
We will 
use the 
notation $\pi$ 
to denote a  projection.
In general,  given two domains
$A$  and $B$, we will denote by $\pi_{A}$ and $\pi_{B}$
the natural projections of the product $A\times B$ to the factors $A$ and $B$, respectively.

\section{Introduction}

 Let $f\colon X \to X$ be 
dominant rational self-map
of a complex projective manifold,
or more generally
a dominant meromorphic self-map of a compact 
K\"ahler manifold of dimension $k$. Let $\omega$ be a
K\"ahler
form on $X$.
For any $0\leq s \leq k$ 
one can define the sequence
$\lam^+_{s,n}:= \| (f^n)_* (\omega^{k-s})\|_X$ 
of the masses of the
positive closed $(k-s,k-s)$-currents
$(f^n)_* (\omega^{k-s})$,
and the \emph{dynamical degree of order $s$ of $f$} as
\[
    \lam_s^+ :=
\limsup_{n\to\infty}
( \lam^+_{s,n})^{1/n}.\]
For cohomological reasons,
$\omega^{k-s}$
could be replaced by any smooth 
form 
 or
some positive closed current 
in the same cohomology class. In particular, the sequence $(\lam^+_{s,n})_{n\in\mathbb N}$
detects the volume growth of $s$-dimensional subvarieties (whenever they exist) under the action of $f^n$,  for  $n\in\mathbb N$.

It turns out that the sequences $(\lam^+_{s,n})_{n\in\mathbb N}$ are (almost)
sub-multiplicative, hence the $\limsup$ in the definition above 
is actually a limit
\cite{DS_regularization, DS_borne, Guedj05, RS97, Veselov92}, see also
\cite{Dang20,Truong20}. 
For a precise behaviour of these sequences in a number of settings, 
see also
\cite{BFJ08, DF21, FW12, Nguyen06, Truong14} and references therein.
It is also a consequence of the fundamental Khovanskii-Teissier inequalities that 
the sequence of the dynamical degrees $\lam^+_s$
is log-concave, i.e., 
the function $s\mapsto \log \lam^+_{s}$ is concave
\cite{DN_mixed, Gromov, Khovanskii_geometry, Teissier_index}. An immediate consequence of this property is that there exists
$1\leq p\leq k$
 such that
 \begin{equation}\label{e:intro-monotonicity}
\lam_0^+ \leq \lam^+_1
 \leq
{\ldots}
\leq \lam^+_p \geq
{\ldots}
 \geq  \lam^+_{k-1} 
\geq \lam^+_k.
\end{equation}
When $X$ is projective, this means that the growth rate 
under the action of $f$ of the volumes
of $s$-dimensional analytic subsets, for $s\leq p$, 
dominates that of $(s-1)$-dimensional analytic subsets,
and the reversed property is true for $s> p$.
Moreover, the so-called \emph{algebraic entropy}
$\log \lam^+_p$ 
of $f$ is 
larger than or equal to the (topological) 
entropy of the system
\cite{DS_regularization, DS_borne,Gromov03,Vu21}, see also
\cite{DTV10, Yomdin87}.

Let us stress that the proof of the log-concavity of the sequence of the degrees
 $\{\lambda_s^+\}_{0\leq s\leq k}$,
and as a consequence
that of 
\eqref{e:intro-monotonicity},
deeply relies on the algebraic setting and on cohomological arguments
(Hodge-Riemann theorem).
In particular, it breaks down
when considering non-compact or local situations, even when $\lam_k^+$ is the degree of maximal value (i.e., when the system is somehow geometrically expanding). We address this problem in this paper, with new tools coming from pluripotential theory.

\medskip

We consider in this paper invertible {horizontal-like maps
in any dimensions
\cite{DNS,DS1} and polynomial-like maps in any dimensions \cite{DSallure, DS10}.
\emph{Horizontal-like maps} are essentially holomorphic maps, defined on
some bounded (convex, for simplicity) subset 
$D$ of $\mathbb C^k$, that have an expanding
behaviour in $p$ directions
and contracting behaviour in the remaining $k-p$ directions.
Such expansion and contraction are of global nature, and these maps are in general not uniformly hyperbolic.
Assuming $D= M\times N$
(with $M\Subset \C^p$ and $N\Subset \C^{k-p}$
bounded convex domains), 
the map $f$ sends a vertical open subset of $D$
to a horizontal one
and, roughly speaking,
the vertical (resp. horizontal)
part of the boundary of the first to the vertical (resp. horizontal)
part of the boundary of the second.
As a particular case,
when $p=k$ the set $N$ reduces to a point and one recovers the notion of \emph{polynomial-like maps}, proper holomorphic maps
of the form $f\colon U\to V$, for some open bounded
subsets $U\Subset V \Subset \mathbb C^k$, with $V$ convex \cite{DSallure,DS10}.

Horizontal-like and polynomial-like maps can be seen
as the building blocks of larger systems and, in particular, 
give a good setting to study local dynamical problems in larger dynamical systems. Small perturbations of such maps still belong to  these classes (up to slightly shrinking the domain of definition), hence 
we get large classes of examples, and the families are
infinite-dimensional. As examples, perturbations of
lifts to $\mathbb C^{k+1}$
of holomorphic endomorphisms of $\mathbb P^k (\mathbb C)$ give examples of polynomial-like maps.
Perturbations of complex H\'enon
automorphisms of $\mathbb C^2$ \cite{BLS93, FS92, Sibony99} give horizontal-like maps with $k=2$ and $p=1$. Such maps were for instance considered in \cite{Dujardin04}. More generally,
we call
any
invertible horizontal-like map
a \emph{H\'enon-like map}.

\medskip

Given a H\'enon-like map or a polynomial-like map, one can introduce dynamical degrees as above. Denoting again by $\omega$ 
the standard
K\"ahler form on $\mathbb C^k$, one can roughly
define
(see Section \ref{s:degrees} for the formal definition)
\[
\lam^+_s 
:= \limsup_{n\to\infty} (\lam^+_{s,n})^{1/n},
\quad \mbox{ where }
\quad
\lam^+_{s,n}:= \|(f^n)_* \omega^{k-s}\|_{M'\times N'}.
\]
Here $M'\Subset M$ and $N'\Subset N$ are open convex sets slightly smaller than $M$ and $N$ respectively. In fact, 
we can show that $\lambda_s^+$ is independent of the choice of $D':=M'\times N'$, see Lemma \ref{l:indep-deg-smooth}. Because of the geometry of the problem,
these definitions
are 
only given
for $0\leq s \leq p$. 
On the other hand, 
for H\'enon-like maps, since $p<k$, one can also define the remaining degrees as
$\lam^-_s (f) := \lam^+_{s} (f^{-1})$
for $0\leq s \leq k-p$
(since $f^{-1}$ 
is
a
``vertical-like" map
 with $k-p$ expanding directions).

Observe that, a priori, the sequences $(\lam^+_{s,n})_{n\in\mathbb N}$ above need not be sub-multiplicative this time.
More importantly,
the lack of a
Hodge
theory means that, a priori, the resulting degree may change if one replaces $\omega^s$ with the integration on a given analytic set of dimension $k-s$, or more generally with a 
(positive closed) current of bi-degree $(s,s)$. Hence, 
for $0\leq s \leq p$
it is natural to also introduce the
degree
\[
d_s^+ := \limsup_{n\to \infty} (d^+_{s,n})^{1/n},
\quad \mbox{  where }
\quad 
d_{s,n}^{ +}:= \sup_{S}\|(f^n)_* (S)\|_{M'\times N}
\]
and $S$ runs over the set 
of all horizontal positive closed currents of bi-dimension $(s,s)$
and of mass 1 on $D':=M'\times N'$. These definitions are also independent of the choice of $D'$. 
As for $\lambda_s^-$, 
for  H\'enon-like maps
we can define
the remaining dynamical degrees as $d^-_s(f):=d^+_{s}(f^{-1})$ for every
$0\le s\le k-p$.

\medskip

The following theorems are our main results, which in particular
answer \cite[Question 6.3]{DNS}.
More complete and precise 
versions
 of Theorems \ref{t:main-intro}
and \ref{c:PL-intro}
are given in
Propositions
\ref{p:lambda} and \ref{p:PL-lambda}
and
Theorems \ref{t:degrees} and \ref{t:PL-degrees}.

\begin{theorem}\label{t:main-intro}
 Let 
$1\leq p<k$ be integers and
$f$ be 
 a H\'enon-like map
from a vertical open subset of a  bounded convex 
domain $D= M\times N \subset \mathbb C^p \times \mathbb C^{k-p}$
to a horizontal open subset of $D$. Then, 
the sequences 
$\{\lam^+_s\}_{0\leq s \leq p}$, $\{\lam^-_s\}_{0\leq s \leq k-p}$,
$\{d^+_s\}_{0\leq s \leq p}$, and $\{d^-_s\}_{0\leq s \leq k-p}$
satisfy
\[
  \lam^+_0 \leq \lam^+_1 \leq 
 \ldots
  \leq \lam^+_p=\lambda^-_{k-p}\geq
  \ldots\geq \lambda^-_1\geq\lambda^-_0
 \]
and
\[1= d^+_0 \leq d^+_1 \leq \ldots \leq d^+_p=d^-_{k-p}\geq\ldots\geq d^-_1\geq d^-_0.
\]
Moreover,  we have $\lambda^+_p = d^+_p \in \mathbb N$.
\end{theorem}

Since $\log d^+_p$ is equal to the topological
entropy
$h_t (f)$
of $f$ \cite{DNS}, 
we deduce the following immediate consequence of Theorem \ref{t:main-intro}.

\begin{corollary}
Let $p,k,f$ 
be as in Theorem \ref{t:main-intro}. Then, 
for all $0\leq s^+ \leq p$ and $0\leq s^- \leq k-p$   
we have
\[
\log \lam_{s^{\pm}}^{\pm} \leq \log d_{s^{\pm}}^{\pm} \leq \log d_p^{+} =
h_t (f).
\]
\end{corollary}

 As we will see, the proof 
of Theorem \ref{t:main-intro} can be applied  also
in
the case of polynomial-like maps, 
 even if these maps are in general not invertible.

\begin{theorem}\label{c:PL-intro}
 Let $k\geq 1$ be an integer and take open sets $U\Subset V\Subset \C^k$  with $V$ convex. Let $f\colon U\to V$ be a polynomial-like map of topological degree $d_t$.
Then the sequences 
$\{\lam^+_s\}_{0\leq s \leq k}$ and $\{d^+_s\}_{0\leq s \leq k}$
satisfy
\[\lambda_0^+ \leq \lambda_{1}^+ \leq \ldots \leq \lambda_{k}^{+}=d_t 
\quad 
\mbox{ and} 
\quad
1= d^+_0 \leq d_{1}^+\leq\ldots \leq d_{k}^{+}= d_t.\] 
 In particular, all
 the
 dynamical degrees
  of $f$
 are smaller than or equal to $d_t$.
\end{theorem}
\medskip

As
far as we are aware of, 
the only (non trivial)
cases
where the monotonicity of the
 dynamical
degrees could be established
until now
are of algebraic nature, and such monotonicity is a consequence of the 
 log-concavity property mentioned above. In a nutshell,
in order to establish the monotonicity of the sequence
$\{d_{s}^+\}_{0\leq s \leq p}$,
given a (horizontal positive closed) 
current  $S$ 
of bi-dimension $(s,s)$
with $s<p$, one needs to find another current,
of bi-dimension $(s+1,s+1)$ whose mass 
growth under iteration bounds the mass growth of $S$ under iteration. 
Although tricky, this is not a big problem when $S$ is smooth
(which essentially gives the monotonicity of the sequence
$\{\lam_{s}^+\}_{0\leq s \leq p}$).
The main problem arises when $S$ is not smooth, and already
in the case where $S$ is given by the integration on an analytic set. 
Even considering an analytic set containing the first one, it is not clear at all why the iterates of the first should behave nicely inside the iterates of the second. 

\medskip

Our solution to the problem can be roughly explained as follows. Given a (horizontal positive closed) current $S$ in $D$,
we first construct a ``holomorphic" 
family of positive closed currents 
$S_\theta$
parametrized by $\theta\in \mathbb D$ and with $S_0 \geq S$ \footnote{
For technical reasons, we often need to reduce slightly the domain $D$, and the estimates that we obtain at any fixed $n$ may depend on the chosen domain. On the other hand, we will show that the limit objects do not depend on such choice. The possible bad behaviour of the currents near the boundary of their domain of definition is a source of technical difficulty in all the paper. For the sake of simplicity, we do not specify this change of domain in this Introduction.}.
We then consider all these currents as
the slices of a unique current
$\mathcal R$,
of bi-dimension $(s+1,s+1)$,
on the space $D\times \mathbb D$, using the slices $D\times \{\theta\}$. The candidate to the role of current of bi-dimension $(s+1,s+1)$ in $D$ would then be $R:=(\pi_D)_* (\mathcal R)$, where $\pi_D\colon D\times \D \to D$ is the natural projection.
Two difficulties arise here. First, we need to make sure that $R$ is well-defined and horizontal in $D$. In order to do this, we suitably modify $\mathcal R$ in the space $D\times \mathbb D = M \times N \times \mathbb D$, in order to make it become \emph{horizontal in} $M \times (N\times \mathbb D)$, i.e., to have support contained in  
$M\times K$, for some compact subset $K$
of $N\times \mathbb D$. This makes both the projection
$(\pi_D)_* (\mathcal R)$ well-defined (since now the projection 
$\pi_D$ 
is proper on the support of $\mathcal R$), and horizontal there (since the projection of the support  of $R$ on $N$ is relatively compact). In order to do this, we exploit some results in the theory of 
the
$d, \overline \partial$, and $dd^c$
equations.
Observe that, in order to get all these controls, it is crucial to work with the extra flexibility given by (positive closed) currents, and not only with analytic subsets.

\medskip

Once $R$ is  well-defined, we still need to make sure that the growth of the mass of $(f^n)_* (R)$
dominates the growth of the mass of
$(f^n)_*(S)$. In order to get this, we need to pay extra attention, and get further estimates, during the construction of $\mathcal R$ and $R$. More precisely, we make sure that the family of deformations $S_\theta$, and the family of the \emph{slices} $ \mathcal R_\theta$
of $ \mathcal R$ with
$D\times \{\theta\}$
are sufficiently continuous in a suitable sense.
Studying the 
growth of the mass of $(f^n)_* (R)$ in $D$
amounts to study the mass growth of $(F^n)_* ( \mathcal R)$,
where $F:=(f,\mathrm{id})$ on 
$D\times \mathbb D$. We prove that the sequence 
of functions
$\phi_n$ on $\mathbb D$, where $\phi_n (\theta)$ is the mass of $(f^n)_* ( \mathcal R_\theta)$
(suitably normalized),
is bounded
with respect to a suitable norm (the DSH norm). A now-classical theorem by Skoda then implies that a large growth of this sequence at $\theta =0$ must imply a large growth of this sequence for $\theta$
sufficiently close to $0$. Going back to currents, this implies a bound (from below)
on the growth of the mass of 
$(f^n)_* ( \mathcal R_\theta)$, hence of
$(F^n)_* ( \mathcal R)$, and hence of $(f^n)_* (R)$. The assertion then follows.

\subsection*{Organization of the paper}
In Section \ref{s:horizontal} we
give the main construction, that we will use in the sequel, to produce a current of bi-dimension $(s+1,s+1)$ in a product space from one of
bi-dimension
$(s,s)$. The construction gives a good control on the support and norms. In Section \ref{s:degrees}
we define the dynamical degrees
for horizontal-like maps, and we give their first properties. The
two chains of inequalities in Theorems \ref{t:main-intro} and \ref{c:PL-intro} are proved in Sections \ref{s:monotone-smooth} and \ref{s:monotone-general}, respectively.

\subsection*{Acknowledgments}
The authors would like to thank the National University of Singapore and the University of Lille for the warm welcome and the excellent work conditions.

This project has received funding from
 the French government through the Programme
 Investissement d'Avenir
 (I-SITE ULNE /ANR-16-IDEX-0004,
 LabEx CEMPI /ANR-11-LABX-0007-01,
ANR QuaSiDy /ANR-21-CE40-0016,
ANR PADAWAN /ANR-21-CE40-0012-01)
and
from 
the NUS
and MOE through the grants
A-0004285-00-00
and 
MOE-T2EP20120-0010.

\section{Deformations of horizontal positive closed currents}\label{s:horizontal}

We fix in this section
integers $1\leq p \leq k$ and convex open bounded subsets $M'' \Subset M' \Subset M\Subset \mathbb C^p$ and $N'' \Subset N' \Subset N \Subset \mathbb C^{k-p}$ and set $D:= M\times N$,
$D' := M' \times N'$, and
$D'':= M''\times N''$.
Observe that, when $p=k$, we have $N=N'=N''$  and these sets reduce to a single point.
We denote by $\pi_M$ and $\pi_N$ the natural projections of $D$ on $M$ and $N$, respectively, and use similar notations for the projections of $D'$ to $M'$ and $N'$ and of $D''$ to $M''$ and $N''$.
We say that a subset $E\subset M\times N$
is \emph{horizontal} in $M\times N$
 if $\pi_N (E)\Subset N$
 and
 \emph{vertical} if $\pi_M (E)\Subset M$. A current in $M\times N$
 is horizontal (resp. vertical) if its support is horizontal (resp. vertical).
 These notions naturally generalize to subsets of and currents on
  other product spaces. 
  Note that when $p=k$
  any current in $M\times N$ is horizontal, since $N$ is a single point.

\medskip

Consider any
convex open sets $M^\star$ and $N^\star$ with $M''\Subset M^\star\subseteq M$ and  $N''\Subset N^\star\subseteq N$. For $0 \leq s \leq p$, we denote  by $\mathcal{H}_s(D^\star)$   
the set of all horizontal positive closed currents of bi-dimension $(s,s)$ of finite mass on  $D^\star:=M^\star\times N^\star$.
We consider the semi-distance 
on 
$\mathcal{H}_s(D^\star) $ 
given by 
\begin{equation}\label{eq:dist}
\dist_{D^\star}
(S,S'):=\sup_\Omega |\langle S-S',\Omega\rangle|, \end{equation}
where $\Omega$ is a real smooth vertical 
$(s,s)$-form
on  $M''\times N^\star$
whose $\Cc^1$-norm is at most 1. 
A similar semi-distance can be defined on the set $\mathcal V_{s}(D^\star)$   
of all vertical positive closed currents of bi-dimension $(s,s)$ of finite mass on 
$D^\star$, for $0\leq s\leq k-p$, by testing
against
real smooth horizontal $(s,s)$-forms
on 
$M^\star\times N''$
 whose $\Cc^1$-norm is at most 1.

\medskip
The main result of this section is the following technical theorem, which will be used in Section \ref{s:monotone-general}.
It gives the deformation of a horizontal positive closed current $S$ on $D$ as described at the end of the Introduction.

\begin{theorem} \label{t:family} 
Let $S$ be a horizontal positive closed current of bi-dimension $(s,s)$ 
and of mass $1$ on $M\times N$
with support contained in $M\times N''$,
for some
$0\leq s\leq p-1$. Then, there exist
a positive closed current $\Rc$ of bi-dimension $(s+1,s+1)$ on $M'\times N \times \D$
and a constant $c>0$ independent of $S$
such that
\begin{enumerate}
\item[{\rm (i)}] $\Rc$ is smooth outside $\pi_{\mathbb D}^{-1}(0)$; 
\item[{\rm (ii)}] the slice
$\Rc_\theta:=\langle \Rc,\pi_\D,\theta\rangle$ is  well-defined
as a current of $M'\times N$ of bi-dimension $(s,s)$
and of mass at most $1$ for every $\theta\in\D$, and $\|\mathcal R\|\leq 1$;
\item[{\rm (iii)}] $S\leq c\Rc_0$ on $M'\times N$;
\item[{\rm (iv)}] $\dist_{D'} (\Rc_\theta,\Rc_{\theta'})\leq |\theta-\theta'|$ for all $\theta, \theta'\in\D$;
\item[{\rm (v)}] 
$\Rc$ is horizontal in $M' \times (N \times \mathbb D)$, with horizontal support in $M'\times (N'\times \D_{1/2})$.
\end{enumerate}
In particular, $(\pi_{M'\times N})_*\mathcal R$ is  well-defined and $\|(\pi_{M'\times N})_*\mathcal R\|\le 1$.
\end{theorem}

Observe that the quantity $\dist_{D'} (\mathcal R_\theta, \mathcal R_{\theta'})$ in the fourth item is  well-defined 
by \eqref{eq:dist}
since, for all $\theta \in \mathbb D$,
the current $\mathcal R_\theta$ is 
horizontal on $M'\times N'$ by the fifth item.

\medskip

We refer to \cite{Federer} for the general theory of slicing of currents, and to \cite{DS1,Dujardin04}
in the particular case of horizontal positive closed currents. 
When  well-defined, the slice
$\langle \mathcal R, \pi_\D, \theta\rangle$ can be seen as the intersection
current
$\mathcal R \wedge [\pi^{-1}_\D (\theta)]$.
In particular, it is a current of bi-dimension $(s,s)$ on $M'\times N \times \D$, supported 
by $M'\times N \times \{\theta\}$,
that
we can identify  with
a current of 
bi-dimension $(s,s)$ on $M'\times N$, see also
\cite{AB93}.
In our case,
since $\mathcal R$ is smooth outside of $\pi_\D^{-1} (0)$, for $\theta \neq 0$
the slice
$\langle\mathcal R, \pi_\D, \theta\rangle$
is equal to  the restriction of $\mathcal R$ to $\pi_\D^{-1} (\theta)$.

\medskip

 Although, a priori, the first property in the statement will not be needed in the sequel, we will
 use smooth deformations 
 in order to obtain the other properties.
  In particular,
the proof of Theorem \ref{t:family} uses
the following lemma, which relies on the solution of the
$d$ and
$\bar \partial$ equations
by means of integral formula, see for instance \cite{BT82, DNS, HL84, HL88, Rudin80}.

\begin{lemma}
\label{l:ddcpsi}
Let $m\geq 2$ and $0\leq l \leq m-1$ be integers.
  Let $U,U',V,V'$ 
  be
   open convex domains of $\C^m$
   with $V' \Subset V \subseteq U'\Subset U$.  Let $\Psi$ be a  positive closed current of bi-degree $(m-l,m-l)$ on $U$, and 
   assume the $\Psi$ is smooth on $V$.
  Then there exists a negative $L^1$ form $\mathcal U_\Psi$
  of bi-degree  $(m-l-1,m-l-1)$ on $U'$, which is smooth on $V$,
   and a positive constant $c$
  (depending on the domains but  independent of $\Psi$)
  such that
\begin{equation}\label{eq:normineq}
    dd^c\mathcal U_\Psi=
    \Psi \mbox{ on } U', \quad
    \quad
     \|\mathcal U_\Psi\|_{U'}\le c\|\Psi\|_{U}, 
     \quad 
     \mbox{ and }
       \quad \|\mathcal U_\Psi\|_{\mathcal C^2 (V')}\le c\|\Psi\|_{\mathcal C^2 (V)}.
 \end{equation}  
\end{lemma}

We will also need the following basic lemma, which will be used several times in the paper.

\begin{lemma}\label{l:general-smooth-bound-compact}
Let
$m\geq 2$, $1\leq p\leq m-1$, and $0\leq l \leq p$
 be integers.
Let $U\subset \mathbb C^p$ and $V\subset \C^{m-p}$ be bounded convex open domains. For every compact subset $K\Subset U\times V$ there exists
 a smooth horizontal positive closed  
 $(m-l,m-l)$-form
 $\Omega_K$
 on $U\times V$
  such that
 $\Omega_K$ is strictly positive on $K$.
\end{lemma}

\begin{proof}
Fix an open convex set
$V'\Subset V$ such that $K \subset  U\times V'$. 
 Take $z\in K$. Since  $l$ satisfies
  $0\le l\le p$, we can find an 
  $l$-dimensional complex plane $\Pi_z$
  passing through $z$ and
  contained in $U\times \{z\}$, so that
  it does
  not intersect $\overline{U}\times \partial V'$. 
  By using a convolution, we can
  average small perturbations of $[\Pi_z]$
  to obtain a closed positive horizontal $(m-l,m-l)$-form
  $\Omega_z$ on  $U\times V'$,
  which is strictly positive at $z$. By continuity, 
  $\Omega_z$
  is actually strictly positive on 
  a neighbourhood of $z$.
  By taking a finite sum of such $\Omega_z$'s,
  we can construct a positive closed horizontal form 
  $\Omega$
  on $U\times V'$ which is strictly positive
  in a neighbourhood of $K$.
   The assertion follows.
\end{proof}

\proof[Proof of Theorem \ref{t:family}] 
We fix $M^\star$ and $N^\star$ 
convex open sets satisfying $M'\Subset M^\star\Subset M$ and $N''\Subset N^\star\Subset N'$,
and set $D^\star := M^\star \times N^\star$. 
Fix $r>0$ sufficiently small so that
the $3r$-neighbourhood of 
$M'$ (resp. $M^\star, N'', N^\star$)
is 
contained in 
$M^\star$ (resp. $M,N^\star,N'$).
Denote by $\mathbb B$  the unit ball of $\mathbb C^k$
and let 
$\mathbb D_2$
be the disc centred at 0 and of radius $2$ in $\mathbb C$.
For $a \in \mathbb B$ and
every $\theta \in \D_2$ define 
the holomorphic automorphism
$h_{a,\theta}:\mathbb{C}^k
\to \mathbb{C}^k$
as
$$h_{a,\theta}(z)=z+r\theta a$$
and 
the holomorphic submersion
$H_{a}: \C^k \times \D_2\to \C^k $
as
     $$H_{a}(z,\theta):=h^{-1}_{a,\theta}(z)=z-r\theta a.$$
    Define $\mathcal{S}^{a}:=(H_{a}^* S)_{| D^\star\times \D_2}$. Then $\mathcal{S}^{a}$ defines a current
     on $ D^\star\times \D_2$ of 
     bi-dimension  $(s+1,s+1)$.
     Since $N$ is convex  and $h_{a,\theta}$ is close to the identity
     (by the choice of $r$), 
     for every $\theta \in \D_2$
     the set $(\supp \mathcal{S}^{a})\cap \pi_{\D_2}^{-1}(\theta)$ 
     is a horizontal set in $ M^\star\times (N^\star\times \{\theta\})$,
     where $\pi_{\D_2}\colon D^\star \times \D_2\to \D_2$ is the natural projection. Hence, the slice
       $(\mathcal{S}^{a})_\theta := \langle \mathcal S^{a}, \pi_{\D_2}, \theta\rangle$ is a horizontal current  of bi-dimension  $(s,s)$ in $M^\star\times N$, supported on $M^\star\times N^\star$
     (where we identify $M^\star \times N \times \{\theta\}$ with $M^\star\times N$).
     Observe that, with this identification, we have $(\mathcal S^a)_\theta = 
     (h_{a,\theta}^{-1})^* S$ on $D^\star$
     and, in particular, $(\mathcal{S}^{a})_0=S$ on $D^\star$.
     Moreover, the dependence $\theta \mapsto (\mathcal{S}^{a})_{\theta}$ is continuous
     for the weak topology of currents.

\medskip

 Let $\rho(a)$ be a 
 fixed smooth 
 positive
  form of maximal degree 
 with compact support in $\B$
 and of integral 1. It defines a probability measure on $\mathbb B$.
  Set
  $$\mathcal{S}:=\int_{\B} \mathcal{S}^{a} \rho(a).$$
  Then $\mathcal S$ is a current on $M^\star \times N\times \D_2$, supported on $D^\star \times \D_2$.
  For $\theta \in \D_2$,
  we denote 
    by $\mathcal S_\theta:=\langle\mathcal S, \pi_{\D_2}, \theta\rangle$
  the slice current of $\mathcal S$ by $\pi_{\D_2}^{-1}(\theta)$, and we identify it to
  a (positive closed) current on $D^\star$
    of bi-dimension $(s,s)$.

\medskip

\noindent \textbf{Claim.}
The current $\mathcal S$
satisfies the following properties:
  \begin{enumerate}
 \item $\mathcal S_0=S$ on $M^\star\times N$;
 \item $\mathcal S_{\theta}$ 
is smooth for all $\theta\in \D_2\setminus\{0\}$;
\item $\|\mathcal S\|_{\mathcal C^2}\leq \tilde c$ outside $M^\star \times N \times \D_{1/10}$ and $\|\mathcal S\|\le \tilde{c}$ on $M^\star \times N \times \D_{2}$;
\item
 $\dist_{D'} (
 \mathcal S_\theta,
 \mathcal S_{\theta'})
 \leq \tilde{c}|\theta-\theta'|$
for all $\theta, \theta' \in \D_2$;
  \end{enumerate}
where $\tilde c$ is a positive constant
independent of $S$.

\medskip

Observe that, as $\mathcal S_\theta$
is supported on $M^\star \times N^\star$ for all $\theta \in \D_2$, we can see its restriction to $M'\times N^\star$
as 
 a (horizontal) current on $M'\times N'$. Hence, the quantity $\dist_{D'}(
 \mathcal S_\theta,
 \mathcal S_{\theta'})$
 in the last item 
 is well-defined
 by \eqref{eq:dist}.

  \begin{proof}[Proof of the Claim]
  It follows
  from the definition that $\mathcal S_0=S$ on $M^\star\times N$.
  To prove the second item,
we show that the coefficients of  $\mathcal S_\theta$ are smooth.
First, observe that
\[\mathcal S_\theta = \int_\B \mathcal (S^a)_\theta \rho(a)
\quad 
\mbox{ for all } \theta \in \D_2.\]

 For any 
two given  
multi-indices
$I:=(i_1,\dots,i_{k-s})$ and $J:=(j_1,\dots,j_{k-s})$
with $i_l,j_l\in\{1, \dots, k\}$ for all $1\leq l\leq k-s$,
let $\gamma(z)$,  
$\gamma^a_{\theta} (z)$, 
and $\gamma_\theta (z)$
be the coefficients 
of $dz_I\wedge d\bar{z}_J:=dz_{i_1}\wedge\dots\wedge dz_{i_{k-s}}\wedge d\bar{z}_{j_{1}}\wedge\dots\wedge d\bar{z}_{j_{k-s}}$
in  ${{S}}$, 
$(\mathcal{S}^{a})_{\theta}$, 
and ${\mathcal S}_\theta$, respectively.
Consider the change 
of coordinates on $\C^k$
given by the translation
\begin{equation}\label{e:changecoor}
    z \mapsto w= h_{a,\theta}^{-1} (z, \theta) = z-r\theta a.
\end{equation}
  Since
  $(\mathcal S^a)_\theta=(h^{-1}_{a,\theta})^*S$,
  we have
  $\gamma^{a}_\theta (z)= 
   \gamma(w)$. Hence, we have
  $$\gamma
  _\theta (z) =
  \int_{w \in \C^k} 
  \gamma
   (w) \rho\left(\frac{z-w}{r\theta}\right).$$
  Thus, $\gamma_\theta$ is smooth for $\theta\ne0$. It follows that
  ${\mathcal S}_\theta$ is smooth for $\theta\in\D_2\setminus\{0\}$. The second assertion follows.

\medskip

 We now prove the third item.
 In the variables $\zeta=(\zeta',\zeta_{k+1}):=(z,\theta)$ with $\zeta':=(\zeta_1, \dots,\zeta_k)$, let $\sigma^a(\zeta)$ and $\sigma(\zeta)$ be the coefficients of  $d\zeta_I\wedge d\bar{\zeta}_J:=d\zeta_{i_1}\wedge...\wedge d\zeta_{i_{k-s}}\wedge  d\bar{\zeta}_{j_1}\wedge...\wedge d\bar{\zeta}_{i_{k-s}}$ in $\mathcal{S}^a$ and $\mathcal{S}$ respectively,  where 
$I:=(i_1,\dots,i_{k-s})$ and $J:=(j_1,\dots,j_{k-s})$
with $i_l,j_l\in\{1, \dots, k+1\}$ for all $1\leq l\leq k-s$. 
 Consider this time
the change of coordinates on $\C^{k+1}$ given by
$$(\zeta',\zeta_{k+1})\mapsto (w,\zeta_{k+1})=(\zeta'-r\zeta_{k+1}a, \zeta_{k+1}).$$
Then, we have
\begin{equation}\label{eq:sigmazeta}
    \sigma (\zeta) =
  \int_{w \in \C^k} 
  \sigma^a
    (w,\zeta_{k+1}) \rho\left(\frac{\zeta'-w}{r\zeta_{k+1}}\right).
\end{equation}
 Since $\mathcal{S}^{a}=(H_{a}^* S)_{| D^\star\times \D_2}$,  we can see that $\sigma^a(w,\zeta_{k+1})$ is independent of  the coordinate $\zeta_{k+1}$.  So the right-hand side  of \eqref{eq:sigmazeta} (and hence $\sigma (\zeta) $) is smooth outside $\zeta_{k+1}=0$.
Hence, $\mathcal{S}$ is smooth outside $\pi^{-1}_{\bD_2}(0)$.  As the estimates are uniform out of a neighbourhood of 
$\{\zeta_{k+1}=0\}=\{\theta=0\}$, this gives the proof 
of the first part of the 
item.
For the 
second part,
 observe that
  $$\mathcal{S}=(\pi_{D^\star \times \D_2})_*\left(\rho(a)\wedge \left.H^*(S)\right|_{D^\star\times\bD_2\times \B}\right) $$
   where $H:\bC^k\times \bD_2\times\B\to\bC^k$ is defined as $H(z,\theta,a)=H_{ a}(z,{\theta}),$
    and $\pi_{ D^\star \times \D_2}:D^\star\times \bD_2\times\bC^k\to D^\star\times \bD_2$ is the natural projection. Since $H$ is a submersion, the masses of $\rho(a)\wedge H^*(S)$ on $D^\star\times \bD_2\times\B$ 
  and
   of its push-forward
    to $D^\star\times \D_2$
   are bounded above by a constant 
    which only depends on $\rho$ and the considered domains. This
   completes the proof of the third assertion in the claim.

\medskip

Let us prove the last item.
Let $\Omega$ be a smooth real 
$(s,s)$-form 
with vertical support in $M''\times N'$ 
and $\mathcal{C}^1$-norm 
less than or equal to 1. 
Then, for all $\theta \in \D_2$, we have
\begin{equation}
\label{e:S-Omega-theta}\langle \mathcal S_{\theta}, \Omega \rangle
=\int_{a \in \B} \langle (\mathcal{S}^{a})_{\theta},\Omega\rangle \rho(a)
=\int_{a \in \B} 
\langle S,(h_{a,\theta})^*\Omega\rangle \rho(a)
=\langle S, \Omega_\theta\rangle,
\end{equation}
where we set 
\[\Omega_\theta:=\int_{a \in \B} (h_{a,\theta})^*(\Omega)  \rho(a).
\]
Since $\Omega$ is
a vertical form with vertical support in
$M''\times N'$, 
$\Omega_\theta$ is  well-defined as a vertical form on $M^\star \times N^\star$. As $S$ is supported on $M\times N''$, the pairing in last term in \eqref{e:S-Omega-theta} is  well-defined.

\medskip

Since $\Omega$ is smooth and real, and $\|\Omega\|_{\mathcal{C}^{0}(M''\times N')}\leq 1$,
there exists a constant
$c_1$ independent of $\Omega$
such that
\[-c_1
\omega^{ s}\|\Omega_{\theta_1}-\Omega_{\theta_2}\|_{\mathcal{C}^0(D^\star)}\le  \Omega_{\theta_1}-\Omega_{\theta_2}\le c_1 \omega^{s} \|\Omega_{\theta_1}-\Omega_{\theta_2}\|_{\mathcal{C}^0(D^\star)}
\mbox{ for all } \theta_1, \theta_2 \in \D_2.\]
Moreover,
since $\|\Omega\|_{\mathcal C^1 (M''\times N')}\le 1$
there exists also 
 a constant
$c_2$
(depending on $\rho$ but independent of $\Omega$)
such that 
\[\|\Omega_{\theta_1}-\Omega_{\theta_2}\|_{\mathcal{C}^0 (D^\star)}\le {c}_2|\theta_1-\theta_2|
\mbox{ for all } \theta_1, \theta_2 \in \D_2.\]
 Since $S$ has mass 1, we have
$$|\langle S_{\theta_1}-S_{\theta_2}, \Omega \rangle|=|\langle S, \Omega_{\theta_1}-\Omega_{\theta_2} \rangle|\le c_1\|\Omega_{\theta_1}-\Omega_{\theta_2}\|_{\mathcal C^0 (D^\star)}\le \tilde{c}|\theta_1-\theta_2|,$$
where $\tilde{c}=c_1c_2$
 is independent of $\Omega$.
So, we have $\mathrm{dist}_{D'}(S_{\theta_1},S_{\theta_2})\le \tilde{c}|\theta_1-\theta_2|$. The proof
of the claim is complete. 
\end{proof}

The Claim implies that the current
$\delta \mathcal{S}$ satisfies properties (i)-(iv) in the statement of Theorem \ref{t:family} for any constant $\delta>0$ small enough.
We now need to modify this current in order to satisfy also
the last property. Observe that
$\pi_N(\supp  \Sc \cap (M^\star \times 
 N
\times \overline{\D}_{1/4}))\Subset N^\star$ by the choice of $r$.

\medskip

Since $\mathcal S$ is  positive 
and closed,
by using Lemma \ref{l:ddcpsi}
for $\Psi = \mathcal S$ on $U'= D'\times \D$ and for $ U:= M^\star\times N \times\D_2$,  $m=k+1$ and $l=s+1\le k$, we see that
there exists a current $\mathcal{U}_{\mathcal S}$ on $D' \times \D$ such that $dd^c \mathcal{U}_\mathcal{S}=\Sc$ on $D'\times \D$.
Moreover, $\mathcal U_\mathcal{S}$ 
is smooth on 
$D' \times (\D\setminus \{0\})$, with good $\mathcal C^2$ estimates on compact subsets of $D' \times (\mathbb D \setminus \{0\})$.

\medskip

Denote $W:= N' \times \D$.  Let $0\le\chi\le 1$ be a non-negative 
smooth
cut-off function, horizontal in $M'\times W$,
which is
equal to 1 on a neighbourhood of $M'\times N^\star\times\overline{\D}_{1/4}$ and
vanishes on $M'\times (N_1 \times \D_{1/3})^c$ for some
 open convex set
 $N_1$ such that $N^\star\Subset N_1\Subset N'$.
Define, on $M'\times W$,
$$
\begin{aligned}
\tilde \Sc :=dd^c(\chi \mathcal U_\Sc) &
=dd^c \chi \wedge\mathcal U_\Sc
+ d\mathcal U_\Sc\wedge d^c\chi
+d\chi\wedge d^c \mathcal U_\Sc+\chi dd^c \mathcal U_\Sc\\
&
=dd^c \chi \wedge\mathcal U_\Sc+
d\mathcal U_\Sc \wedge d^c\chi +d\chi\wedge d^c \mathcal U_\Sc+\chi \Sc.
\end{aligned}$$ 
Then $\tilde{\mathcal S}$
is horizontal and closed
in $M'\times W$. It is smooth outside $\pi_{\mathbb D}^{-1}(0)$.
Moreover, we have
$\tilde {\mathcal{S}} =\mathcal{S}$ on $M'\times N^\star\times {\D}_{1/4}$
(since $dd^c \chi = d\chi = d^c \chi =0$ and $\chi=1$ there)
and
$\supp \tilde{\mathcal{S}}\subset M'\times \overline{N_1}\times 
\overline \D_{1/3}$
since $\chi$ vanishes on $M'\times (N_1 \times \D_{1/3})^c$.
In particular, $\tilde{\mathcal S}$
satisfies (v). On the other hand,
 $\tilde \Sc$ is not necessarily positive on $M'\times 
 {N_1}\times 
(\overline \D_{1/3} \setminus 
{\D}_{1/4})
$. However, it is bounded from below by some smooth negative form independent  of $S$ because the $\mathcal{C}^2$-norm of $\mathcal{U}_\mathcal{S}$ is bounded there  by a constant independent 
 of $S$.
We now construct a 
smooth horizontal
positive
closed  $(k-s,k-s)$-form $\Omega_+$ on $M'\times N'\times \D_{1/2}$
such that $\tilde \Sc + \Omega_+$ is (horizontal, closed and) positive, and has good support and
norm estimates.

\medskip

Set $F:= \overline{N_1}\times 
\overline \D_{1/3}$.
By Lemma \ref{l:general-smooth-bound-compact} applied with $m=k+1$, $l=s+1$ (which is possible since $s\leq p-1$ by assumption),
$U=M$, $V= N'\times \D_{1/2}$, 
 and
$K= \overline{M'}\times F$,
there exists a smooth
horizontal positive closed
$(k-s,k-s)$-form $\Omega_+$ on $M\times N' \times \D_{1/2}$
which
is strictly positive on $\overline{M'}\times F$.
 Since  $\tilde {\mathcal{S}} =\mathcal{S}$ on  $M'\times N^\star\times {\D}_{1/4}$, and outside this set 
 the
$\mathcal{C}^0$-norm of $\tilde {\mathcal{S}}$ is bounded by
a constant independent  of 
$S$, by taking 
large enough
constants $b_1, b_2>0$
(independent of $S$)
 we can see that
the current
$\Rc:= b_2^{-1} (\tilde {\mathcal S} + b_1 \Omega_+)$
satisfies the requirements on the statement. 
This concludes the proof.

\endproof

\section{Dynamical degrees of horizontal-like maps}\label{s:degrees}

We again fix in this section 
 integers $1\leq p\leq k$ and  
a bounded convex domain $D= M\times N \subset \mathbb C^p \times \mathbb C^{k-p}$.  
We also set  $\partial_v D:=\partial M\times \overline{N}$ (resp. $\partial_h D:= \overline{M}\times \partial N$)  and we call it the \textit{vertical} (resp. \textit{horizontal}) boundary of $D$.
We denote by
$\pi_1$ and $\pi_2$
the first and the second projections of $D\times D$ on its factors, respectively.

\begin{definition}
\label{d:HLM}
A \emph{horizontal-like map} $f$  on $D$
is a holomorphic map whose graph $\Gamma\subset D\times D$ satisfies the following properties:
\begin{enumerate}
    \item $\Gamma$ is a submanifold of $D\times D$ of pure dimension $k$ (not necessarily connected);
    \item $(\pi_1)_{|\Gamma}$ is injective;
    \item $(\pi_2)_{|\Gamma}$ has finite fibers;
    \item $\overline \Gamma$ does not intersect ${\partial_v D} \times \overline D$ and
    $\overline D \times  {\partial_h D}$.
 
\end{enumerate}

\end{definition}

Observe in particular that $f$ 
 does not need to be defined on the whole 
$D$, but only on the
vertical open subset
$\pi_1 (\Gamma)$ of $D$.
Likewise, the image of $f$
is the horizontal subset
$\pi_2 (\Gamma)$
of $D$. We will write 
$D_{v,1}= \pi_1 (\Gamma)$
and $D_{h,1}= \pi_2 (\Gamma)$
in the
following.
More generally, for all $n\geq 1$
we 
consider the iterate $f^n = f\circ \dots \circ f$ ($n$ times). These are
also horizontal-like maps.
We denote by $D_{v,n}$
and $D_{h,n}$
the domain and the image of $f^n$. Observe that the sequences
$(D_{v,n})_{n\geq 1}$
and $(D_{h,n})_{n\geq 1}$ are decreasing.

\begin{remark}\label{r:PL}
 In the case where $p=k$,
as $N$ is a point,
Definition \ref{d:HLM} amounts to
simply
considering proper
holomorphic maps $f\colon M_0 \to M$, where $M_0:= f^{-1} (M)$
is open, $M_0\Subset M$,
and $M\Subset \C^k$ is a convex open set
(we set in this case $\partial N=\varnothing$ by convention, to keep the condition (iv) consistent).
These are the so-called \textit{polynomial-like maps}, see for instance \cite{DSallure}.
A polynomial-like map $f\colon M_0 \to M$ defines a ramified covering over  $M$. We usually call the degree $d_t$ of the covering 
the \emph{topological degree} of $f$.
\end{remark}

\medskip

The \emph{filled (forward) Julia set}  $\mathcal{K}_+$ 
of $f$
is defined as
\[
\mathcal{K}_+ := \{ z \in D \colon 
f^n(z) \mbox{ is defined for all } n\geq 0
\}.
\]
Note that $\mathcal{K}_+=\cap_{n\ge 1} D_{v,n}$.
Define also 
$\mathcal{K}_-:=\cap_{n\ge 1} D_{h,n}$.
Both $\mathcal K_+$ and $\mathcal K_-$ are non-empty, and
we have
$f^{-1}(\mathcal{K}_+)
= \mathcal{K}_+$
 and
$f(\mathcal{K}_-)=\mathcal{K}_-$.
Observe that $\mathcal{K}_+$ (resp. $\mathcal{K}_-$)
is a vertical (resp. horizontal)
closed
subsets of $D$.
In particular, $\mathcal{K}_+\cap \mathcal{K}_-$
is a compact subset of $D$.

\medskip
 We call any invertible horizontal-like map (i.e.,  any horizontal-like map such that $(\pi_2)_{|\Gamma}$ is also injective)   a \textit{H\'enon-like map}. 
In the rest of this section, $f$  will be either a H\'enon-like map or a polynomial-like map in the sense of Remark \ref{r:PL}.
 The sets $M',M'', N',N''$ will be as in the Notation at the beginning of the paper.  We will assume in what follows that
$M'$ and $M''$ (resp. $N'$ and $N''$) are chosen
sufficiently close
to $M$ (resp. $N$) so that
\begin{equation}\label{e:hp-inclusion}
D_{v,1} \subset M'' \times N
\quad \mbox{ and } \quad
D_{h,1} \subset M \times N''. 
\end{equation}
In particular, 
we can replace $D$ with $D'$ or
$D''$ and still get
horizontal-like maps. We denote by
$D'_{v,n}$, $D'_{h,n}$, $D''_{v,n}$, and
$D''_{h,n}$ the domains and images of the restrictions
of $f^n$ to $D'$ and $D''$, respectively. Note that, 
when $f$ is a polynomial-like map,
$N, N'$, and $N''$
consist of a single point and
we have $D_{v,n}=f^{-n}(M)$ and
$D_{h,n}=M$ (and, similarly, $D'_{v,n}=f^{-n}(M')$, $D'_{h,n}=M'$, $D''_{v,n}=f^{-n}(M'')$,  and $ D''_{h,n}=M''$). In this case, condition \eqref{e:hp-inclusion} means that $f^{-1}(M)\subset M''$.

\medskip

The following lemma gives some basic compactness estimates that follow from the inclusions
\eqref{e:hp-inclusion}
and that we will need in the sequel.
Recall that the operators $f_*$ and $f^*$ are defined as
\[
f_* := ((\pi_2)_{|\Gamma})_* \circ ((\pi_1)_{|\Gamma})^* \quad \mbox{ and }
\quad 
f^* = ((\pi_1)_{|\Gamma})_* \circ ((\pi_2)_{|\Gamma})^*.
\]
The operator  $f_*$
is continuous on horizontal currents, and
the operator $f^*$
is continuous on vertical currents. 
Recall that we only consider  H\'enon-like maps and polynomial-like maps.

\begin{lemma}\label{l:collection-shift}
Let $M,M',M'',N,N',N''$, and $f$ 
be as above.
Then, there exists a constant $A$ (depending on the domains  and $f$) such that 
   \begin{enumerate}

     \item[{\rm (i)}]\label{eq:shift1}  
   $(f^{n})^*(\omega^s)\le A (f^{n-1})^*(\omega^s)$ on $D'_{v, n}$
    for all $0\leq s \leq k$
    and $n \geq 1$;

  \item[{\rm (ii)}]\label{eq:ineqAS} $(f_* (S))_{|D'} \in \mathcal H_s(D') 
   \mbox{ and } \|f_* (S)\|_{M'\times N}\leq  A \| S\|_{M''\times N}$ 
 for all $0\leq s \leq p$ and $S \in \mathcal{H}_s(D')$.
   \end{enumerate}
 
\end{lemma}

\begin{proof}
Set 
${A}:=\max\limits_{0\le s\le k}\|f^*(\omega^s)\|_{\mathcal{C}^1(D'_{v,1})}$. Then  
we have
$f^*(\omega^s)\le {A} \omega^s$
on  $D'_{v,1}$
for  all
$0\le s\le k$. The first assertion follows  
 from the inclusion
$D'_{v,n}\subset D'_{v,1}$. 

Take now $S\in\mathcal{H}_s(D')$, for some $0\le s\le p$. 
 Using  \eqref{e:hp-inclusion}, we have
 \begin{align*}
     \|f_* (S)\|_{M'\times N}&=\int_{M'\times N}f_* (S)\wedge\omega^s\le\int_{M'\times N}f_*\left( S\wedge f^*(\omega^s)\right)  \\
     &=\int_{f^{-1}(M'\times N)} S\wedge f^*(\omega^s)\le\int_{M''\times N} S\wedge f^*(\omega^s)\\
     &=\int_{M''\times N'} S\wedge f^*(\omega^s)\le {A}\| S\|_{M''\times N}.
 \end{align*}
In the last inequality, 
we used the fact that $S\in\mathcal{H}_s(D')$ and the first assertion with $n=1$. Since $f$ is  horizontal-like, $(f_*(S))_{|D'}$ is a horizontal current on $D'$. By the above inequality its mass is finite and hence $(f_*(S))_{|D'}\in \mathcal{H}_s(D')$. This completes the proof.
 \end{proof}

We now define
quantities to measure the growth of the mass of the currents $(f^n)_* (S)$
(for $S$ horizontal and of dimension up to $p$) and $(f^n)^*(R)$
(for $R$ vertical and of dimension up to $k-p$). The case of
dimension $p$ and $k-p$, respectively, is given by  Lemma \ref{l:degree-p} below.
Recall that, by \cite[Theorem 2.1]{DS1} 
for any $S \in \mathcal H_p(D)$,
the \emph{slice measure}
 $\langle S,\pi_{M'}, z\rangle$
is well-defined for any $z\in M'$ and its mass, denoted by $\|S\|_h$, is independent of $z$
and is equal to
$\langle S, \pi_{M}^* (\Omega_M)\rangle$
for every smooth probability measure $\Omega_M$
with compact support in $M'$.
Similarly, 
if $R\in\mathcal V_{k-p}(D)$ 
the 
slice measure
$\langle R,\pi_{N'}, w\rangle$
is well-defined for any $w\in N'$ and its mass, denoted by $\|R\|_v$ is independent of $w$ and equal to $\langle R, \pi_N^*(\Omega'_N)\rangle$
for every smooth probability measure $\Omega'_N$
with compact support in $N'$.

\begin{lemma}
  If  $S \in \mathcal H_p(D)$ is supported in $M\times N'$, then
$\|S\|_{D'}\lesssim \|S\|_{h}\lesssim \|S\|_{D'}$.  Similarly, if $R\in\mathcal V_{k-p}(D)$ is supported in $M'\times N$ then $\|R\|_{D'}\lesssim \|R\|_{v}\lesssim \|R\|_{D'}$.
\end{lemma}

\begin{proof}
It is enough to prove
the assertion for $S\in \mathcal H_p(D)$. The inequality $\|S\|_h \lesssim \|S\|_{D'}$ is a consequence of the fact that $(\pi_{M'})_* (S_{|D'}) = \|S\|_h \cdot [M']$, which gives $\|S\|_{D'}\geq \|(\pi_{M'})_* (S_{|D'})\| \gtrsim \|S\|_h \cdot |M'|$, where
$|M'|$ denotes the Lebesgue measure of $M'$.
The proof of the other inequality is essentially given in \cite[Lemma 3.3.3]{DSallure}, see also \cite[Lemma 2.5]{DS10}. We recall here the idea of the proof.

Let $\omega_M$ and $\omega_N$ denote the standard K\"ahler forms on $M$ and $N$, respectively. Then, $\omega = \omega_M+\omega_N$ is the standard K\"ahler form on $D$.
For any $[p\times (k-p)]$-matrix $A$, consider the projection $\pi_{M,A}\colon (\C^p\times \C^{k-p})\to \C^p$ defined as $\pi_{M,A} (z,w)= z + Aw$.
If the entries of $A$ are sufficiently small
(depending on $M,M'$, and $N$), 
$\pi_{M,A}$ is well-defined on $D'$, with image in $M$. By the first part of the proof, the slice
$\langle S, \pi_{M,A}, z\rangle$ is then  well-defined for every $z \in M'$.
Its mass is independent of $A$, and in particular equal to $\|S\|_h$.
Hence, the integral $\int_{D'} S \wedge \pi^*_{M,A} (\omega_M^p)$ is well-defined and bounded by a constant times $\|S\|_h \cdot |M|$ for every $A$.
To conclude, it is enough to observe that $\omega^p$ can be bounded on $D'$ by a finite sum of 
forms 
$c_i \cdot \pi_{M,A_i}^* (\omega_{M}^p)$, where the entries of $A_i$ can be taken small as above, and the $c_i$'s are positive.
As a consequence, we have
\[
\|S\|_{D'}=
\int_{D'} S \wedge \omega^p \leq \sum_i c_i \int_{D'} S \wedge \pi^*_{M,A_i} (\omega_M^p)
\lesssim \|S\|_h.
\]
The assertion follows.
\end{proof}

\begin{lemma}[{\cite[Proposition 4.2]{DS1}}]\label{l:degree-p}
The operators $f_*\colon \mathcal H_p(D')\to \mathcal H_p(D')$ 
and
$f^* \colon \mathcal V_{k-p}(D')\to \mathcal V_{k-p}(D')$
are well-defined and continuous. Moreover, there exists an integer $d\geq 1$ such that
$\|f_* (S)\|_h = d \|S\|_h$ for
all $S \in \mathcal H_p$ and
$\|f^* (R)\|_v= d \|R\|_v$
for all $R \in \mathcal V_{k-p}(D')$.
\end{lemma}

We call the integer $d$ as in the last statement the \emph{main dynamical degree} of $f$. Note that when $f$ is a polynomial-like map (see Remark \ref{r:PL}), since $\mathcal{H}_k(D')$ is the set of constant functions, the main dynamical degree is the topological degree $d_t$ of $f$.

\medskip

We now introduce 
some invariants in order to measure the growth of currents of arbitrary dimension and degree, as above. The first is an adaptation of the (smooth) dynamical degrees for polynomial-like maps
introduced in \cite{DSallure}.

\begin{definition}[Dynamical degrees of type I]
\label{d:degrees-smooth}
Let $M,M',N,N'$, and $f$ be as above.
For $0 \leq s \leq p$, we define the dynamical degree $\lam^+_s$ as
$$\lambda_s^+ :=
\limsup_{n\to\infty} (\lambda_{s,n}^+)^{1/n}, \quad \text{ where} \quad \lambda_{s,n}^+:= 
\|(f^n)_*(\omega^{k-s}_{|D'_{v,n}})\|.
$$
For $0 \leq s \leq k-p$, we define the dynamical degree $\lam^-_s$ as
\[
\lam_{s}^-:= \limsup_{n\to\infty} (\lambda_{s,n}^-)^{1/n}, \quad \text{ where}
\quad \lam_{s,n}^- := \|(f^n)^*(\omega^{k-s}_{|D'_{h,n}})\|
.\]
\end{definition}

The mass in the definition of $\lam^{+}_{s,n}$
(resp.  $\lam^{-}_{s,n}$)
is taken on 
$M'\times N$
(resp. $M\times N'$), 
or equivalently on 
the image $D'_{h,n}$ (resp. the domain
$D'_{v,n}$) of $f^n$.

\begin{remark}\label{r:lambda-f-1}
Note that for polynomial-like maps we only consider $\lambda^+_s$ with $0\le s\le k$. 
For 
H\'enon-like maps,  we can see that $\lam_s^+ (f) = \lam_{s}^-(f^{-1})$ for all $0\leq s \leq p$
and $\lam_{s}^- (f) = \lam_{s}^+ (f^{-1})$ for all
$0\leq s\leq k-p$,
where the degrees of $f^{-1}$ can be defined 
reversing the role of $f_*$ and $f^*$ and using the fact that $f^{-1}$ is a vertical-like map with $k-p$ expanding directions, see also \cite[Section 3]{DNS}.
\end{remark}

\begin{lemma}\label{l:indep-deg-smooth}
 The definitions of $\lam^+_s$
and $\lam_s^-$
are independent of the choice of $M'$ and $N'$ as above. Moreover, we
have $\lam_{p}^+= \lam_{k-p}^-=d,$ the main dynamical degree.

\end{lemma}

\begin{proof}
For the proof in the case of polynomial-like maps we refer
 to
\cite[Lemma 2.3]{DS10}. So, we
assume
that
$f$ is a H\'enon-like map.  
 By Remark \ref{r:lambda-f-1}, it is enough to
prove the statement for $\lambda^+_s$, for $0\leq s \leq p$. 
 We fix convex open sets  $M''\Subset M'$ and $N''\Subset N'$ as above  and $0\leq s \leq p$,
and we denote by
$\tilde{\lambda}^+_{s,n}, \tilde {\lambda}^+_s$ 
the dynamical degrees 
of type I defined using
$D''_{v,n}$ 
instead of
$D'_{v,n}$
in Definition \ref{d:degrees-smooth}. 
Since
$$\|(f^n)_*(\omega^{k-s}_{|D''_{v,n}})\|
\le
\|(f^n)_*(\omega^{k-s}_{|D'_{v,n}})\|
$$
it is clear that, for all $n\geq 0$, we have
$\tilde{\lambda}_{s,n}^+\le \lambda_{s,n}^+$ for $0\leq s \leq p$, hence
$\tilde{\lambda}_s^+\le \lambda_s^+$
 for every $s$ as before.
So, we only need to prove the reverse inequality.

\medskip

Recall that we are assuming that 
$D_{h,1} \subset M\times N''$ and $D_{v,1} \subset M''\times N$,  see \eqref{e:hp-inclusion}. 
In particular, Lemma \ref{l:collection-shift}
holds and since $f^n: D'_{v,n}\to D'_{h,n}$ is bijective
we have 
\begin{equation}\label{e:domains}
    f^{-1}(D_{h,n}')=f^{ n-1}(D_{v,n}')\subset D_{h,n-1}''\cap D'_{v,1}
\end{equation}
for 
all
$n\ge 2$. Indeed, 
 for all $n\geq 2$,
since $D_{h,n}'\subset D_{h,1}'$  we have that
$f^{-1}(D_{h,n}')\subset D_{v,1}'$. Thanks to \eqref{e:hp-inclusion} 
we also have
$D_{v,n}'=f^{-n}(D')=f^{-n}(M'\times N'')\subset f^{-n+1}(M''\times N'')=D_{v,n-1}''$. So \eqref{e:domains} follows.

Using \eqref{e:domains} and the first assertion in Lemma \ref{l:collection-shift} it follows that, for 
 all
$n\geq 2$ and $0\leq s \leq k- p$, we have
\begin{align*}
    \|(f^n)_*(\omega^{k-s}_{|D'_{v,n}})\|
        & =
    \int_{D'_{h,n}} (f^n)_* (\omega^{k-s}_{|D'_{v,n}})\wedge \omega^{s}\\    
           & =
    \int_{f^{-1}(D'_{h,n})} (f^{n-1})_* (\omega^{k-s}_{|D'_{v,n}})\wedge f^*(\omega^{s})\\  
     &
    \lesssim \int_{f^{-1}(D'_{h,n})}
    (f^{n-1})_*(\omega^{k-s}_{|D'_{v,n}})\wedge \omega^{s}
           \\
     &\le \int_{D_{h,n-1}''}
    (f^{n-1})_*(\omega^{k-s}_{|D'_{v,n}})\wedge \omega^{s}
       \\
    &    \le\int_{D''_{h,n-1}}
    (f^{n-1})_*(\omega^{k-s}_{|D''_{v,n-1}})\wedge \omega^{s}= 
          \|(f^{n-1})_*(\omega^{k-s}_{|D''_{v,n-1}}) \|,
\end{align*}
where in the last inequality we used the inclusion $D'_{v,n}\subset D''_{v,n-1}$. Hence we have $\lam_{s,n}^+ \lesssim \tilde \lam_{s,n-1}^+$, which implies that $\lambda_s^+\le\tilde{\lambda}_s^+$ and the independence of $\lam_s^+$ from the domains. 

\medskip

Finally, let us
show that $\lam_{p}^+=d$. A similar proof also
shows that $\lam_{k-p}^-=d$,  see also Remark \ref{r:lambda-f-1}.
There exists $\alpha_-, \alpha_+ \in \mathcal H_{p}(D')$ 
such that
$\|\alpha_{\pm}\|_h=1$ and
\[
\alpha_- \lesssim \omega^{k-p}
\mbox{ on } D'
\quad \mbox{ and }
\quad 
\omega^{k-p}
\lesssim \alpha_+ \mbox{ on } D''
\]
(where the existence of $\alpha_+$ follows from Lemma \ref{l:general-smooth-bound-compact}). 
We apply Lemma \ref{l:degree-p} with $S=\alpha_-$ and obtain that $\lambda^+_p\ge d$. 
By the first part of the proof,
the dynamical degree of type I 
is independent of the choice of $D'$. So we can replace $D'$ with $D''$ in the definition of $\lambda^+_p$. Applying
Lemma  \ref{l:degree-p}  with $\alpha_+$ gives $\lambda^+_p\le d$. Hence,  we have $\lambda^+_p=d$.
The proof is complete.
\end{proof}

We will also consider the following 
notion 
of
dynamical degrees,
which was
introduced in \cite{DNS},
see \cite{DS10} for the earlier definition in the setting of polynomial-like maps.

\begin{definition}[Dynamical degrees of type II]
\label{d:degrees-currents}
For $0\leq s\leq p$, the dynamical degree $d_s^+$ is defined by
$$d_s^+:=\limsup_{n\to\infty} (d_{s,n}^+)^{1/n} \quad \text{ where } \quad d_{s,n}^+:=\sup_S \|(f^n)_*(S)\|_{M'\times N},$$
and $S$ runs over the set 
$\mathcal{H}_s^{(1)}(D') $ 
of all horizontal positive closed currents of bi-dimension $(s,s)$ of mass 1 on $D':=M'\times N'$.

For $ 0\leq s\leq k-p$,
the dynamical degree $d_s^-$ 
is defined by
$$d_s^-:=\limsup_{n\to\infty} (d_{s,n}^-)^{1/n} \quad \text{ where } \quad d_{s,n}^-:=\sup_R \|(f^n)^*(R)\|_{M\times N'},$$
and $R$ runs over the set 
$\mathcal{V}_{s}^{(1)}(D')$ 
of all vertical positive closed currents of bi-dimension $(s,s)$ of mass 1 on $D':=M'\times N'$.
\end{definition}

\begin{remark}\label{r:d+(f)=d-(f-1)}
Similarly as
for dynamical degrees of type I, 
for polynomial-like maps we
have $p=k$, hence we
only
 need to
consider $d_s^+ $ with $0\le s\le k$. When $f$ is a  H\'enon-like map,
 we can define the degrees of type II for the vertical-like map $f^{-1}$ and
we have
$d_s^+ (f) = d_{s}^- (f^{-1})$
for $0 \leq s \leq p$ and 
$d_s^- (f) = d_{s}^+ (f^{-1})$ for $0\leq s \leq k-p$.
\end{remark}

The following lemma is proved
in \cite[Lemma 2.6]{DS10} in the case of polynomial-like maps and in
\cite[Lemma 3.5]{DNS} 
 in the case of H\'enon-like maps.

\begin{lemma}\label{l:degreeind} 
The dynamical degrees $d_s^+$ and $d_s^-$ do not depend on the choice of $M'$ and $N'$. Moreover, we have $d_0^+ =d_0^- =1$
and $d_p^+= d^-_{{ k-p}}=d$. 
\end{lemma}

In particular, it follows from Lemma \ref{l:degree-p} that
\[d^n \lesssim d^+_{p,n} \lesssim d^n 
\quad 
\mbox{ and }
\quad 
d^n \lesssim
d^{-}_{{ k-p},n} \lesssim d^n \quad \mbox{ as }n\to\infty.\]
\begin{remark}\label{r:dsge1}
By considering a current $S$ given by the integration on a horizontal analytic subset of dimension $0\leq s\leq p$
which intersects $\mathcal K_+$
we can see that $d_s^+\geq 1$ because $(f^n)_*(S)$ satisfies the same property for all $n$, and hence these currents have mass bounded from below. 
Recall that  for polynomial-like maps we only
 have to consider $d_s^+$  with $0\le s\le k$.
When $f$ is a H\'enon-like map,
we can apply 
 the above argument for $f^{-1}$.  By Remark \ref{r:d+(f)=d-(f-1)}, we get that  $d_s^-\geq 1$ for $0\le s\le k-p$. 
\end{remark}

\begin{remark}\label{rmk:ineq-lambda-d}
Take $0\leq s \leq p$. 
By Lemma \ref{l:general-smooth-bound-compact}, 
one can bound $\omega^{k-s}_{|D''}$ from above with
a smooth horizontal positive closed $(k-s,k-s)$-form $\Omega$
on $D'$ (we used here that $s\leq p$).
Therefore, it is easy to deduce that $\lambda_s^+\leq d_s^+$
if we use $D''$ to compute $\lam_s^+$
and $D'$ to compute $d_s^+$, see Lemmas \ref{l:indep-deg-smooth} and \ref{l:degreeind}.
Similarly,
one can see that $\lam^-_s \leq d^-_s$ for all $0\leq s \leq k-p$.
\end{remark}

\begin{lemma}\label{l:submult-degrees}
The sequences $(d_{s,n}^+)_{n\in \N}$
(for $0\leq s \leq p$) 
and $(d_{s,n}^-)_{n\in \N}$ (for $0\leq s \leq k-p$)
are sub-multiplicative, i.e., we have
$d^{\pm}_{s,n+m}\le  d^{\pm}_{s,m}d^{\pm}_{s,n}$
for all $n,m\geq 0$.
In particular,
we have
$$d_s^\pm=\lim_{n\to\infty} (d_{s,n}^\pm)^{1/n}=\inf_{{ n\ge 1}} (d_{s,n}^\pm)^{1/n}.$$
\end{lemma}

\begin{proof}
 We prove the assertion for $d_s^+$ for a given $0\leq s \leq p$, the argument is the same for $d_s^-$ for $0\leq s \leq k-p$,  see also Remark \ref{r:d+(f)=d-(f-1)}.
 As $d_{s,0}^+=1$ for all $0\leq s \leq p$, we can assume that $n\geq 1$.

 Take $S\in \mathcal{H}_s^{(1)}(D')$. Then, using that $f$ is horizontal-like,
  for all $m\geq 0$
  we have
 \begin{align*}
     \|(f^{n+m})_*(S)\|_{M'\times N}&= \|(f^{m})_*((f^{n})_*(S))\|_{M'\times N}\\
    &\le \|(f^{n})_*(S)\|_{M'\times N} \cdot \|(f^{m})_*(T)\|_{M'\times N}\\
  &  \le  d_{s,n}^+\cdot \|(f^{m})_*(T)\|_{M'\times N},
 \end{align*}
   where $T:= 
   (\|(f^{n})_*(S)\|_{M'\times N})^{-1}
   \cdot
   ((f^{n})_*(S))_{|M'\times N} $
   is a current of mass 1 on $M'\times N$. 
More precisely,
$T$ is a horizontal positive closed current whose support is contained in 
 the horizontal subset
$D_{h,n}'$  of $M'\times N$. By \eqref{e:hp-inclusion},
it is then also a horizontal positive closed current on $M'\times N'$.
   In particular, $T$ belongs to $\mathcal H_s^{(1)}(D')$.
   Therefore, 
    by Definition \ref{d:degrees-currents}, for all $m\geq 0$
   we have
   $\|(f^{m})_*(T)\|_{M'\times N}\le d_{s,m}^+$.
   It follows that
     $$ \|(f^{n+m})_*(S)\|_{M'\times N}  \le  d_{s,m}^+d_{s,n}^+
   \quad   \mbox{ for all } S \in \mathcal H_s^{(1)}(D'),$$
     which implies that $d_{s,n+m}^+\le  d_{s,m}^+d_{s,n}^+$  for all $n\geq 1$ and $m\geq 0$.
 By the classical Fekete lemma, the second assertion is a consequence of the first one.
 \end{proof}

\section{Monotonicity of
dynamical degrees of type I}\label{s:monotone-smooth}

 In this section we prove the
  assertions in Theorems \ref{t:main-intro} and \ref{c:PL-intro}
  involving the dynamical degrees of type I. We first deal with H\'enon-like maps.
Recall that, in this case,
we fix integers $1\leq p<k$, a bounded and convex domain $D= M\times N\subset\C^p \times \C^{k-p}$
and 
the convex open sets $M''\Subset M'\Subset M$
and $N''\Subset N'\Subset N$
are assumed to be sufficiently close to $M$
and $N$
so that
\eqref{e:hp-inclusion} and
Lemma \ref{l:collection-shift} hold.
The following result gives the first 
chain of inequalities in Theorem \ref{t:main-intro}.
   The $\lambda^{+}_{s}$'s
  and 
  $\lambda^{-}_{s}$'s
  in the statement are given by computing
 the
 masses on $M'\times N$ and $M\times N'$
in Definition \ref{d:degrees-smooth}, respectively.

\begin{proposition} \label{p:lambda}
 Let
$1\leq p<k$ be integers, $f$ be
 a H\'enon-like map on 
$D = M\times N \subset \C^p \times \C^{k-p}$,
and 
$M',M'',$ $N',$ $N'',$ $ \lambda^+_s,\lambda^-_s$ be as above.
Then, we have
\[\lambda_{s}^+\leq  \lambda_{s+1}^+  \;\;
\mbox{ for } \;\;
0\leq s\leq p-1
\quad \mbox{ and } \quad 
\lambda_{s}^-\leq  \lambda_{s+1}^- \;\; \mbox{ for } \;\;
0\leq s\leq k-p-1.\]
\end{proposition}

\proof 
 It is enough to
prove the inequalities for the degrees $\lam_{s}^+$. The 
assertions on the degrees $\lambda_s^-$
can be obtained
by replacing $f$  with $f^{-1}$,  see Remark \ref{r:d+(f)=d-(f-1)}.
We fix $0\leq s < p$ in the following.
 Since $1\le p< k$, by Lemma \ref{l:general-smooth-bound-compact}
there exists
a
 smooth horizontal closed $(k-p,k-p)$-form $\Omega$ 
 on $D'$
 satisfying $\omega^{k-p}_{|D''}\lesssim \Omega \lesssim \omega^{k-p}_{D'}$.
 Let $\chi_1$ be a cut-off function equal to 1 on $M''\times N$  and which vanishes
in a neighbourhood of $(M\setminus M')\times N$.
 Let $\chi_2$ be another cut-off function equal to 1 on the support of $\chi_1$ and vanishing 
 in a neighbourhood of $(M\setminus M')\times N$.
 In particular, the supports of both $\chi_1$
 and $\chi_2$ are vertical in $M'\times N$.
 Then, for all $n\geq 1$,  using
  the fact that
$s< p$ we have
\[\begin{aligned}
\|(f^{n+1})_*(\omega^{k-s}_{|D''_{v,n+1}})\|_{M''\times N'}
&= \int_{D''_{h,n+1}}
(f^{n+1})_* (\omega^{k-s})\wedge \omega^s\\
&\lesssim \int_{
D''_{h,n+1}}
(f^{n+1})_*(\Omega\wedge \omega^{p-s}) \wedge \chi_1\omega^{s} \\
&\leq \int_{
D'_{v,n+1}}
\Omega\wedge \omega^{p-s} \wedge (f^{n+1})^*(\chi_1\omega^{s}) \\
&\lesssim \int_{
M'\times N}
\ddc\|z\|^2 \wedge \Omega\wedge \omega^{p-s-1} \wedge (f^{n+1})^*(\chi_1\omega^{s}),
\end{aligned}\]
where $z$ is the standard 
coordinate on $\mathbb C^k$ such that
$\omega = \ddc \|z\|^2$ 
and in the last step we also used the fact that
$D'_{v,n+1}\subset D'_{v,1}\subset M'\times N$, see \eqref{e:hp-inclusion}.
Since $\Omega$ is horizontal 
and closed
and $\chi_1$ is vertical in $M'\times N$,
we can apply Stokes theorem and deduce that
\[\begin{aligned}
\|(f^{n+1})_*(\omega^{k-s}_{|D''_{v,n+1}})\|_{M''\times N'}
&\lesssim  \int_{M'\times N} \|z\|^2 \Omega\wedge \omega^{p-s-1} \wedge (f^{n+1})^*(\ddc\chi_1 \wedge \omega^{s}) \\
&\lesssim \int_{M'\times N}  \Omega\wedge \omega^{p-s-1} \wedge (f^{n+1})^*(\chi_2\omega^{s+1}) \\
&\lesssim \int_{M''\times N}  \Omega\wedge \omega^{p-s-1} \wedge (f^n)^*(\chi_2\omega^{s+1}),
\end{aligned}\]
where in the last step we used again
the
fact that
$D'_{v,n+1}\subset D'_{v,n}
\subset  M''\times N$
 (see \eqref{e:hp-inclusion}),  Lemma \ref{l:collection-shift}, and the fact that 
$(f^{n+1})^*(\chi_2)\le (f^{n})^*(\chi_2)$ on $D'_{v,n+1}$.
Using the fact that $\chi_2$ is supported in $M'\times N$, and the above upper bound for $\Omega$,
we obtain
\[\begin{aligned}
\|(f^{n+1})_*(\omega^{k-s}_{|D''_{v,n+1}})\|_{M''\times N'}
& \lesssim 
\int_{D'}
\Omega\wedge \omega^{p-s-1} \wedge (f^n)^*(\omega^{s+1})\\
& =\int_{D'_{h,n}}  (f^n)_*(\Omega\wedge \omega^{p-s-1}) \wedge \omega^{s+1}\\
& \lesssim \| (f^n)_*( \omega^{k-s-1}_{|D'_{v,n}}))\|.
\end{aligned}\]
This proves the inequality $\tilde \lambda^+_{s,n+1}\lesssim  \lambda^+_{s+1,n}$,
where $\tilde \lambda^+_{s,n}$
is the quantity defined using $D''_{v,n}$ instead of $D'_{v,n}$
in Definition \ref{d:degrees-smooth},
and 
the implicit constant is independent of $n$.
The inequality  $\lambda^+_s \leq \lambda^+_{s+1}$
follows by taking the powers $1/n$ and a $\limsup$
for $n\to \infty$ in the previous one, and applying Lemma \ref{l:indep-deg-smooth}. The proof is complete.
\endproof

The following proposition gives the first part of Theorem \ref{c:PL-intro}. 
\begin{proposition}\label{p:PL-lambda}
     Let $f:U\to V$ be a 
polynomial-like map where $U\Subset V$ is open 
and $V\subset \C^k$ is a bounded convex open set.
Then 
 $\lambda_s^+\leq \lambda_{s+1}^+$ for all $0\leq s\leq k-1.$
\end{proposition}

\begin{proof}

 The proof 
 is similar 
  that
 of Proposition \ref{p:lambda},  hence
 we only sketch it. 
   
   Let $V''$ and $V'$ be convex domains satisfying $U\Subset V''\Subset V'\Subset V$. 
  Note that $f^{-1}(V)=U\Subset V''$.  In the sense of Remark \ref{r:PL},   we have the  identifications
  \[M'\times N=V'=D'_{h,n},
  \quad M''\times N'=V''=D''_{h,n}, \quad D'_{v,n}=f^{-n}(V'), \quad \mbox{and}\quad
  D''_{v,n+1}=f^{-n-1}(V'').\] 
 Recall that in this case $N=N'=N''$ is a single point.
  
 As $p=k$, we have $k-p=0$. By putting $\Omega\equiv1$
   and taking two cut-off functions $\chi_1$ and $\chi_2$ 
  with compact support in $V'$ and  such that $\chi_1=1$ on $V''$ and $\chi_2=1$ on the support of $\chi_1$,  we can repeat the same  computations as in the proof of  Proposition \ref{p:lambda}  and
  obtain
   \begin{align*}
    \|(f^{n+1})_*(\omega^{k-s})\|_{V''}
&\lesssim  \|(f^{n+1})_*(\omega^{k-s-1})\|_{V'}
   \end{align*}
   for every $0\le s\le k-1$,
    where the implicit constant is independent of $n$.
      Since $\lambda_s^+$ is independent of the choice of $V'$ (see 
     Lemma \ref{l:indep-deg-smooth})
   we have $$\lambda_s^+=\limsup_{n\to\infty} \|(f^{n+1})_*(\omega^{k-s})\|_{V''}^{1/n}
\leq  \limsup_{n\to\infty}\|(f^{n+1})_*(\omega^{k-s-1})\|_{V'}^{1/n}=\lambda_{s+1}^+.$$ 
 The proof is complete.
   \end{proof}

\section{Monotonicity of dynamical degrees of type II
}\label{s:monotone-general}

The following theorems answer \cite[Question 6.3]{DNS}
and complete the proof
of Theorems \ref{t:main-intro} and \ref{c:PL-intro}.
We work with the choice of $M',M'',N',N''$
as in the beginning of Section \ref{s:monotone-smooth}. 
In particular, we assume
  that
    \eqref{e:hp-inclusion}
  and Lemma \ref{l:collection-shift} hold.
 As in the previous sections, 
the degrees $d^+_{s,n}$ and $d^+_s$ are computed with respect to $M'$ and $N'$,
as in Definition \ref{d:degrees-currents}.

\begin{theorem} \label{t:degrees}
Let  $ 1\leq p<k$ be integers and $f$ be  a H\'enon-like map
on a
bounded convex domain
$D = M\times N \subset \C^p \times \C^{k-p}$.
Then 
 \[d_s^+\leq d_{s+1}^+  \;\;
 \mbox{ for }
 0\leq s\leq p-1
 \quad 
\mbox{ and } \quad 
 d_{s}^{-}\leq d_{s+1}^{-}\;\;
\mbox{ for } \;\; 0\leq s\leq k-p-1.\]
 In particular, all the dynamical degrees 
 of $f$
 are smaller than or equal to $ d_p^+ = d^-_{k-p} =d$.
\end{theorem}
 The following
is the counterpart of the above statement for polynomial-like maps.
\begin{theorem}\label{t:PL-degrees}
    Let
    $k\geq 1$ be an integer and
    $f:U\to V$ be a 
polynomial-like map of topological degree $d_t$, where $U\Subset V\Subset \C^k$ are open  sets and $V$ is convex.
Then 
 $d_s^+\leq d_{s+1}^+$ for all $0\leq s\leq k-1.$
 In particular, all
 the
 dynamical degrees
  of $f$
 are smaller than or equal to $d_k^+=d_t$.
\end{theorem}

Since $f$ is 
 assumed to be
invertible in Theorem \ref{t:degrees}, it is enough to prove 
only the assertion on the degrees 
$d^+_s$ in that statement.  Because of the similarity of the proofs of the monotonicity of 
 $\{d^+_s\}_{0\leq s \leq p}$ 
in Theorems  \ref{t:degrees} and \ref{t:PL-degrees},
we will give their proofs in parallel. 
 From now on,
we work under the assumptions of 
Theorems \ref{t:degrees} 
and  \ref{t:PL-degrees}  and we fix an $s$ as in those statements.
So, below $f$  is either an invertible horizontal-like map or a polynomial-like map.
In the case of a
polynomial-like map (i.e., for $p=k$), 
we identify 
 $V$ with $M$ and  $N$ to 
a single point, see Remark \ref{r:PL}.

\medskip

Fix an open convex set $M^\star$  with $M''\Subset M^\star\Subset M'$. The quantities that we 
will introduce may depend 
on the domains $M', M^\star, M'',N',N''$,
even if this is not explicitly stated.    Recall that $\mathcal{H}_s(M^\star\times N')$ is the set of all horizontal positive closed currents of bi-dimension $(s,s)$ of finite mass on $M^\star\times N'$. We will need the following basic lemma.

\begin{lemma} \label{l:Lipschitz}
There is a constant $L>0$ such that the action of $f_*\colon \mathcal{H}_s(M^\star\times N') \to \mathcal H_s(M^\star\times N')$ 
is $L$-Lipschitz with respect to the semi-distance 
$\dist_{M^\star\times N'}$ defined as in \eqref{eq:dist}.
\end{lemma}

\proof

Let $\Omega$ be a smooth form with vertical support in $M''\times N'$ and whose $\Cc^1$-norm is at most 1.  Since $f^*(\Omega)$ is also vertical, for $S,S'\in \mathcal{H}_s(M^\star\times N') $ 
we have
$$|\langle f_*(S)-f_*(S'),\Omega\rangle| = |\langle S-S', f^*(\Omega)\rangle| \leq \dist_{M^\star\times N'}(S,S') \|f^*(\Omega)\|_{\Cc^1 (M''\times N')}.$$
Since
$\|f^*(\Omega)\|_{\Cc^1(M''\times N')}$
is bounded by a constant
independent from $\Omega$, the result follows.
\endproof
\begin{remark}
    Note that
   Lemma \ref{l:Lipschitz} is
    still true even when $f$ is 
    a non-invertible horizontal-like map.    On the other hand, it is not clear how to get a version of Lemma \ref{l:Lipschitz} for the action of $f^*$ in the case where $f$ is not invertible. In particular, such a result cannot hold 
in general
because,
as soon as the critical set intersects 
the Julia set $\mathcal J^+ = \partial \mathcal K^+$,
the 
$\mathcal C^1$, Lipschitz, and H\"older norms are not preserved by $f_*$.
\end{remark}

Let $\tilde{d}_{s,n}^+$ be the quantity obtained by replacing $M'$ with $M^\star$ in  the definition of $d_{s,n}^+$ in Definition \ref{d:degrees-currents}.  Recall that the
dynamical degree $d^+_{s+1}$ does not depend on the choice of $M'$ and $N'$, see Lemma \ref{l:degreeind}. The following more precise proposition implies Theorems \ref{t:degrees} and \ref{t:PL-degrees}.

{\begin{proposition} \label{p:degrees}
 Let $f$  be either an invertible horizontal-like map or a polynomial like map as above.
We have $d_{s,n+1}^+\lesssim 
  n
 \tilde{d}_{s+1,n}^+$
 as $n\to \infty$.
\end{proposition}

Recall that we fix $0\leq s\leq p-1$  in all this section.
By the
Definition 
\ref{d:degrees-currents}
of 
$d^{+}_{s,n}$, 
for every $n\in \mathbb N$
there exists $S_{n-1} \in \mathcal H_s^{(1)}(D')$ 
such that
\begin{equation}\label{e:mass}
   \|(f^{n})_*(S_{n-1})\|_{M'\times N}
    \geq
     \frac{1}{2} d^+_{s,n}.
\end{equation}
 It follows from \eqref{e:mass}
 (applied for $n+1$ instead of $n$)
 and Lemma \ref{l:collection-shift}(ii)
 that
\begin{equation}\label{eq:massm''}
    \|(f^{n})_*(S_{n})\|_{M''\times N} \gtrsim 
\|(f^{n+1})_*(S_{n})\|_{M'\times N}\gtrsim d_{s,n+1}^+.
\end{equation}

 For every $n\in \mathbb N$, we apply Theorem \ref{t:family}, 
with $M',M^\star$, and $S_n$ instead of $M, M'$,  and $S$, respectively. This gives
a corresponding
current $\Rc_n$ on $M^\star\times N\times \D$ satisfying the properties 
(i)-(v) of $\mathcal R$
in that theorem. 
Fix a cut-off function
$0\le\chi\le 1$ with vertical support in $M^\star\times N$ 
and equal to 1 on $M''\times N$.
For each $n$,
define the
function $\phi_n:\D\to \mathbb R^+$ as
\begin{equation}\label{e:def-phin}
\phi_n(\theta):=\big\langle (d_{s,n+1}^+)^{-1} (f^n)_*(\Rc_{n,\theta}), \chi\omega^{s}\big\rangle,
\end{equation}
where $\Rc_{n,\theta} = \langle \mathcal R_n, \pi_{\mathbb D},\theta\rangle$
is the slice current of $\Rc_{n}$ for $\theta \in \mathbb D$ and the integral
 in \eqref{e:def-phin}
is taken over $M^\star\times N$.
\medskip

Proposition \ref{p:degrees} will follow
from the following three claims. We keep the
assumptions and the notations as above.
\begin{claim}\label{claim1}
There exists
a constant $\beta>0$, independent  of $n$, such that   $\phi_n(0)\geq 2\beta$
for all $n$.
\end{claim}

\proof
By Theorem \ref{t:family}(iii)  there exists $c>0$ independent of $n$ such that 
\begin{align*}
   \phi_n(0)  &  = \big\langle (d_{s,n+1}^+)^{-1} (f^n)_*(\Rc_{n,0}), \chi\omega^{s}\big\rangle 
      \ge c \, \big\langle (d_{s,n+1}^+)^{-1} (f^n)_*(S_{n}), \chi\omega^{s}\big\rangle
        \\
    & \ge c \,
    \int_{M''\times N} 
       (d_{s,n+1}^+)^{-1} (f^n)_*(S_{n}) \wedge \omega^{s}
       = c\,  (d_{s,n+1}^+)^{-1}  \| (f^n)_* (S_n) \|_{M''\times N}.
\end{align*}
For the second inequality, we have used the fact that 
$\chi\equiv 1 $ on $M''\times N$. Finally,    \eqref{eq:massm''} implies that
the last 
expression is larger than $2\beta$, for some $\beta>0$.
The claim follows.
\endproof

We will see that
$\ddc \phi_{n}$ is a signed measure on $\D$.
When $dd^c \phi_{n}$
is a signed measure, 
we can define its mass 
$\|dd^c \phi_{n}\|$
as the sum of the masses of its positive and negative parts.

\begin{claim}\label{claim2}
The mass
of $\ddc\phi_{n}$ satisfies
$$ \|\ddc\phi_{n}\|\lesssim (d_{s,n+1}^+)^{-1} \tilde{d}_{s+1,n}^+ 
 \quad \mbox{ as } n \to \infty.$$
\end{claim}

\proof 
Recall that we denote by $\pi_{M^\star\times N}$ the projection of $(M^\star\times N) \times \D$
to the factor $M^\star\times N$.
Set 
\[T_n:=(d_{s,n+1}^+)^{-1}\Rc_{n}\wedge \pi_{M^\star\times N}^*((f^n)^* ( \chi\omega^{s})),\]
with $\chi$ as above. 
By  Theorem  \ref{t:family}(v),
the current
$\Rc_{n}$ is horizontal on $M^\star\times (N'\times \D)$. So,
$T_n$ is 
(well-defined and)
compactly supported
in ${M^\star\times N'}\times\D$. In particular,
$\pi_\D$ is proper on the support of $T_n$.
Hence, we have
$$(\pi_\D)_* T_n=
(\pi_\D)_*
\big((d_{s,n+1}^+)^{-1}\Rc_{n}\wedge \pi_{M^\star\times N}^*((f^n)^* ( \chi\omega^{s}))\big)
= \phi_n.$$
 It follows that
$$\ddc \phi_n= \ddc\big((\pi_\D)_* T_n\big)=(\pi_\D)_*(\ddc T_n)= (\pi_\D)_*\big[(d_{s,n+1}^+)^{-1}\Rc_{n}\wedge \pi_{M^\star\times N}^*((f^n)^* ( \ddc \chi \wedge \omega^{s}))\big],$$
where we used that both $\Rc_{n}$ and $\omega$
are closed.

Since $\chi$ is smooth 
there exists a positive constant $c$
such that
$-c\omega \leq \ddc \chi \leq c \omega$
on $M^\star\times N$.
Therefore, we have
\[
\big|
(d_{s,n+1}^+)^{-1}\Rc_{n}\wedge \pi_{M^\star\times N}^*((f^n)^* ( \ddc \chi \wedge \omega^{s}))
\big| 
 \lesssim
(d_{s,n+1}^+)^{-1}\Rc_{n}\wedge \pi_{M^\star\times N}^*\left((f^n)^* (  \omega^{s+1})\right),
\]
where we denoted by $|\cdot|$ the sum of the positive and negative parts of the measure in the left-hand side.
It follows that 
\[
\begin{aligned}
\|\ddc \phi_n\|
&
\lesssim
\|(\pi_{\D})_* \big( (d_{s,n+1}^+)^{-1}\Rc_{n}\wedge \pi_{M^\star\times N}^*((f^n)^* (  \omega^{s+1}))\big)\|_{\D}\\
& = 
(d_{s,n+1}^+)^{-1}
\| \Rc_{n}\wedge \pi_{M^\star\times N}^*((f^n)^* (  \omega^{s+1}))\|_{M^\star\times N \times \D},
\end{aligned}\]
where in the last step we 
 used
 the fact that 
the pushforward operator $(\pi_\D)_*$
preserves the mass of positive measures.
To conclude, we need to bound the last term as in the statement.

\medskip

Since $\Rc_{n}$ is horizontal on $M^\star\times (N'\times \D)$, the projection 
$\pi_{M^\star\times N'}$ is proper on the support of $\Rc_{n}$. 
 It follows that $(\pi_{M^\star\times N})_*(\Rc_{n})$ is well-defined and is a horizontal positive closed current of bi-dimension $(s+1,s+1)$ on $M^\star\times N'$.  Moreover,  by Theorem \ref{t:family}(ii), the mass of  $(\pi_{M^\star\times N})_*(\Rc_{n})$ is less than or equal to $1$}.
 Using again the fact that $\pi_{M^\star\times N}$ preserves the mass of positive measures,
it follows that
\begin{align*}
    \left\|
       \Rc_{n}
     \wedge \pi_{M^\star\times N}^*((f^n)^* (  \omega^{s+1}))\right\|_{M^\star\times N \times \D}
    & =\left\|(\pi_{M^\star\times N})_*\big[
       \Rc_{n}\wedge \pi_{M^\star\times N}^*((f^n)^* (  \omega^{s+1}))\big]\right\|_{M^\star\times N}\\
    &=
       \|(\pi_{M^\star\times N})_*(\Rc_{n})\wedge (f^n)^* (\omega^{s+1})\|_{M^\star\times N}\\
    &
    \leq
       \tilde{d}_{s+1,n}^+,
\end{align*}
where in the last step we used the fact
that 
 $\|(\pi_{M^\star\times N})_* \mathcal R_n\|\leq 1$
and Definition
\ref{d:degrees-currents} 
(recall that we replace $M'$ with $M^\star$ to define 
$\tilde{d}_{s+1,n}^+$).
We deduce from 
the inequalities above and
the definition of $\phi_n$
that
$$\|dd^c\phi_{n}\|\lesssim (d_{s,n+1}^+)^{-1} \tilde{d}_{s+1,n}^+
\quad \mbox{ as } n\to \infty.$$
The proof of the claim is complete.
\endproof

Define now the function $\Phi_n: \mathcal{H}_s (M^\star\times N')\to \R$ by
$$\Phi_n(S):= \big\langle (d_{s,n+1}^+)^{-1} (f^n)_*(S), \chi\omega^{s}\big\rangle,$$
where 
the integral is taken over $M^\star\times N$ and,
 for every $n$,
is  well-defined since $(f^n)_* (S)$
is horizontal
and $\chi$ has vertical support in $M^\star\times N$.

\begin{claim}\label{claim3}
There exist
a constant
$L>0$
independent of $n$
such that
$$|\Phi_n(S)-\Phi_n(S')|\lesssim L^n \, \dist_{M^\star\times N'}(S,S')\quad \mbox{ as } n\to \infty.$$
\end{claim}

\proof  
Since $\chi$ is
a
smooth function with vertical support in $M^\star\times N$,
 by
\eqref{e:hp-inclusion},
$f^*(\chi \omega^s)$ is a vertical smooth form with support in $M''\times N$.
Hence, 
 for all $n\geq 1$,
by Lemma \ref{l:Lipschitz}
we have
\begin{align*}
   |\Phi_n(S)-\Phi_n(S')|& =(d_{s,n+1}^+)^{-1}
|\langle
(f^{n-1})_*(S-S'),f^*(\chi \omega^s)\rangle|\\
&\leq (d_{s,n+1}^+)^{-1}
\|f^*(\chi \omega^s)\|_{\Cc^1(M''\times N')}
{L}^{n-1}\dist_{M^\star\times N'}(S,S')\\
&\lesssim  (d_{s,n+1}^+)^{-1}
{L}^{n}\dist_{M^\star\times N'}(S,S')
\end{align*}
for some positive constant $L$
independent of $n$.
Since  $(d_{s,n+1}^+)^{-1}\lesssim 1$ 
(see Remark \ref{r:dsge1}), 
the assertion follows.
\endproof

We can now 
prove
 Proposition \ref{p:degrees}, which also
concludes the proof of Theorems \ref{t:degrees}
and \ref{t:main-intro}
(for $f$ invertible horizontal-like map)
and Theorems \ref{t:PL-degrees} and \ref{c:PL-intro}
(for $f$ polynomial-like map).

\proof[Proof of Proposition \ref{p:degrees}] 
We keep the notations and definitions as above.
We assume by contradiction that 
\begin{equation*}
\limsup_{n\to\infty} {\frac{d_{s,n+1}^+} 
{n
\tilde{d}_{s+1,n}^+}} =+\infty.
\end{equation*}
Hence, 
there exist  two sequences $n_j\to\infty$ and $c_j\to\infty$ such that
\begin{equation}\label{eq:ds=ds1}
    d_{s,n_j +1}^+= c_j  n_j
       \tilde{d}_{s+1,n_j}^+
    \quad
    \mbox{ for all } j.
\end{equation}

By
 Theorem \ref{t:family}(v),
 we see  that  $\phi_{n_j}(\theta)=0$
   for all
  $j$ and $1/2<|\theta|<1$.
 By
Claim \ref{claim2} 
and \eqref{eq:ds=ds1},
 we also have that
$\|dd^c\phi_{n_j}\|\lesssim (d_{s,n_j+1}^+)^{-1} \tilde{d}_{s+1,n_j}^+= (c_j n_j)^{-1}$
as $j\to \infty$.
Hence, 
the family 
$\big\{ c_j 
n_j \phi_{n_j}\big\}_{j \in \mathbb N}$ 
is a bounded family of $\DSH$ functions on $\mathbb P^1$, i.e.,
differences of quasi-psh functions on $\mathbb P^1$
(see \cite[Appendix A.4]{DS10}).
Indeed, it is enough to apply Lemma \ref{l:DSH} below with $X= \mathbb P^1$ and $\mu$ a smooth probability measure supported on $\mathbb P^1 \setminus \mathbb D$.
 Observe that, by the same lemma, the fact that the family is bounded does not depend on the choice of the measure $\mu$ used to define the norm.

According to Skoda-type estimates \cite{S82}, 
there exist
 positive constants $\alpha,C$,
  independent of $j$,
such that
\begin{equation}\label{eq:skoda}
  \int_{\D} e^{\alpha c_j
   n_j |\phi_{n_j}|} \leq C
  \quad \mbox{ for all } j,  
\end{equation}
where the integral is with respect to 
the Lebesgue measure on $\D$.

Theorem \ref{t:family}(iv) and 
Claim  \ref{claim3} imply that
$$|
\phi_{n_j}(0)-
\phi_{n_j}(\theta)|\lesssim  
L^{n_j}
\dist_{M^\star\times N'}{(\Rc_{n_j,0},\Rc_{n_j,\theta})}\lesssim L^{n_j}|\theta| 
\mbox{ for all } \theta \in \mathbb D.$$
Hence, for some 
 $r_j$ such that
$L^{-n_j}\lesssim r_j \lesssim L^{-n_j}$, by Claim \ref{claim1}
we have
\[
\phi_{n_j}(\theta) \geq \beta 
\quad \mbox{ for } 
\quad 
|\theta|\leq r_j
\quad \mbox{ and all } j \mbox{ sufficiently large}.\] 
This and the above estimate
\eqref{eq:skoda} imply that, for all $j$ sufficiently large,
\[
C  \geq \int_{\D} e^{\alpha c_j 
n_j
|\phi_{n_j}|}
 \ge \int_{\D_{r_j}} e^{\alpha c_j 
n_j
|\phi_{n_j}|}
      \gtrsim 
    e^{\alpha c_j n_j \beta} r_j^2
    \simeq
    e^{\alpha c_j n_j \beta} L^{-2n_j}
    = e^{n_j
    \left(
      \alpha c_j \beta
    -2 \log L\right)}.
    \]
Since $c_j\to\infty $, this gives the desired contradiction, and completes the proof.
\endproof

 \begin{lemma}[{\cite[p.283]{DS10}}]\label{l:DSH}
  Let $\mu$ be a positive measure on a
  compact K\"ahler manifold $X$ which integrates the $\DSH$ functions. Then
  \[
  \|u\|_\mu := \big|\int u \mu\big| + \min \|T^{\pm}\|
  \]
  (where the minimum is taken over all positive closed $(1,1)$-currents $T^{\pm}$ such that $dd^c u = T^+-T^-$)
  defines a norm on $\DSH(X)$.
   Moreover, for all $\mu$ and $\mu'$ which integrate the DHS functions, the norms $\|\cdot\|_\mu$ and $\|\cdot\|_{\mu'}$ are equivalent.
  \end{lemma}

\end{document}